\documentclass[11pt]{amsart}
\usepackage{url, 
	amssymb,setspace, mathrsfs,fontenc, comment}
\usepackage[alphabetic]{amsrefs}
\usepackage{fullpage} 
\usepackage{color}
\usepackage{tikz-cd}
\usepackage[all]{xy}
\usepackage{amsmath}

\usepackage{url, 
	amssymb,setspace, mathrsfs,fontenc}
\usepackage{amsrefs}
\usepackage{graphicx}
\usepackage[mathcal]{euscript}
\usepackage{verbatim}
\usepackage{hyperref}
\hypersetup{backref=true}
\usepackage{mathtools}
\usepackage{tikz}
\usetikzlibrary{chains}

\tikzset{node distance=2em, ch/.style={circle,draw,on chain,inner sep=2pt},chj/.style={ch,join},every path/.style={shorten >=4pt,shorten <=4pt},line width=1pt,baseline=-1ex}

\newtheorem{thm}{Theorem}
\newtheorem{lem}[thm]{Lemma}
\newtheorem{prop}[thm]{Proposition}
\newtheorem{conj}[thm]{Conjecture}
\newtheorem{cor}[thm]{Corollary}

\theoremstyle{remark}
\newtheorem{rem}[thm]{Remark}

\theoremstyle{definition}

\DefineSimpleKey{bib}{myurl}
\newcommand\myurl[1]{\url{#1}}
\BibSpec{webpage}{
	+{}{\PrintAuthors} {author}
	+{,}{ \textit} {title}
	+{}{ \parenthesize} {date}
	+{,}{ \myurl} {myurl}
}
\usepackage{arydshln}

\newcommand{\nc}{\newcommand}

\nc{\ssec}{\subsection}

\nc{\on}{\operatorname}

\nc{\sE}{\mathscr{E}}
\nc{\sF}{\mathscr{F}}
\nc{\sL}{\mathscr{L}}
\nc{\sD}{\mathscr{D}}
\nc{\sA}{\mathscr{A}}

\nc{\cC}{\mathcal{C}}
\nc{\cG}{\mathcal{G}}
\nc{\cV}{\mathcal{V}}
\nc{\CB}{\mathcal{B}}
\nc{\cK}{{k(\!(s)\!)}}
\nc {\K}{\mathcal{K}}

\nc{\cE} {\mathcal{E}}
\nc{\Kl}{\mathrm{Kl}}
\nc{\cO}{\mathcal{O}}
\nc{\cF}{\mathcal{F}}
\nc{\cZ}{\mathcal{Z}}
\nc{\bcZ}{\overline{\mathcal{Z}}}
\nc{\bcB}{\overline{\mathcal{B}}}
\nc{\cD}{\mathcal{D}}
\nc{\cDt}{\mathcal{D}^\times}
\nc{\cH}{\mathcal{H}}
\nc{\bZ}{\mathbb{Z}}
\nc{\bQ}{\mathbb{Q}}
\nc{\bR}{\mathbb{R}}
\nc{\bC}{\mathbb{C}}
\nc{\bQl}{\overline{\mathbb{Q}}_\ell}
\nc{\bQlt}{\bQl^\times} 
\nc{\FG}{\mathrm{FG}}
\nc{\dR}{\mathrm{dR}}

\nc{\uG}{\underline{G}}
\nc{\uc}{\underline{c}}
\nc{\uu}{\underline{u}}
\nc{\cU}{\mathcal{U}}
\nc{\rat}{\mathrm{rat}}
\nc{\Hyp}{\mathrm{Hyp}}
\nc{\Lie}{\mathrm{Lie}}
\nc{\ctheta}{\check{\theta}}
\nc{\nil}{\mathrm{nil}}

\nc{\fF}{\mathfrak{F}}
\nc{\fB}{\mathfrak{B}}
\nc{\fZ}{\mathfrak{Z}}
\nc{\fx}{\mathfrak{x}}
\nc{\fy}{\mathfrak{y}}
\nc{\fb}{\mathfrak{b}}
\nc{\fk}{\mathfrak{k}}
\nc{\fI}{\mathfrak{i}}
\nc{\fj}{\mathfrak{j}}
\nc{\fg}{\mathfrak{g}}
\nc{\fu}{\mathfrak{u}}
\nc{\fl}{\mathfrak{l}}
\nc{\fn}{\mathfrak{n}}
\nc{\cP}{\mathcal{P}}
\nc{\ft}{\mathfrak{t}}
\nc{\fz}{\mathfrak{z}}
\nc{\fc}{\mathfrak{c}}
\nc{\cfc}{\check{\mathfrak{c}}}
\nc{\fh}{\mathfrak{h}}
\nc{\fp}{\mathfrak{p}}
\nc{\bone}{\mathbf{1}}
\nc{\tg}{\mathtt{g}}
\nc{\hfg}{\widehat{\fg}}
\nc{\ch}{\check{\fh}}
\nc{\hP}{\hat{P}}
\nc{\hg}{\widehat{\mathfrak{g}}}
\nc{\gO}{\mathfrak{g}[\![t]\!]}
\nc{\Ug}{\widehat{U}(\mathfrak{g})}
\nc{\dl}{/\!\!/}

\nc{\bGm}{\mathbb{G}_m}
\nc{\bGa}{\mathbb{G}_a}
\nc{\bL}{\mathbf{L}}
\nc{\bK}{\mathbf{K}}
\nc{\bJ}{\mathbf{J}}
\nc{\bI}{\mathbf{I}}
\nc{\bV}{\mathbb{V}}
\nc{\bM}{\mathbb{M}}
\nc{\bP}{\mathbb{P}}
\nc{\bA}{\mathbb{A}}
\nc{\bN}{\mathbb{N}}

\nc {\Q}{\mathrm{Q}}
\nc{\diag}{\mathrm{diag}}
\nc{\diff}{\mathrm{diff}}
\nc{\ev}{\mathrm{ev}}
\nc{\Res}{\mathrm{Res}}
\nc{\Fl}{\mathcal{F}\ell}
\nc{\Ad}{\mathrm{Ad}}
\nc{\ad}{\mathrm{ad}}
\nc{\pr}{\mathrm{pr}}
\nc{\Sl}{\mathfrak{sl}}
\nc{\gl}{\mathfrak{gl}}
\nc{\ra}{\rightarrow}
\nc{\tra}{\twoheadrightarrow}
\nc{\hra}{\hookrightarrow}
\nc{\quo}{\mathopen{ /\!/}}
\nc{\GL}{\mathrm{GL}}
\nc{\SL}{\mathrm{SL}}
\nc{\Sp}{\mathrm{Sp}}
\nc{\SO}{\mathrm{SO}}
\nc{\so}{\mathfrak{so}}
\nc{\PGL}{\mathrm{PGL}}
\nc{\Bun}{\mathrm{Bun}}
\nc{\supp}{\mathrm{supp}}
\nc{\bgamma}{\bar{\gamma}}
\nc{\I}{\mathrm{I}}
\nc{\II}{\mathrm{II}}
\nc{\III}{\mathrm{III}}
\nc{\ab}{\mathrm{ab}}
\nc{\td}{\mathrm{d}}
\nc{\Ht}{\mathrm{ht}}

\nc         {\rar}[1]       {\stackrel{#1}{\longrightarrow}}

\nc{\fa}{\mathfrak{a}}
\nc{\Hit}{\mathrm{Hit}}

\nc{\RS}{\mathrm{RS}}
\nc{\Loc}{\mathrm{Loc}}
\nc{\tLoc}{\widetilde{\mathrm{Loc}}}
\nc{\reg}{\mathrm{reg}}
\nc{\im}{\mathrm{Im}}

\nc{\tp}{\mathfrak{p}}
\nc{\cA}{\mathcal{A}}
\nc{\cY}{\mathcal{Y}}

\nc{\opp}{\mathrm{opp}}
\nc{\Ind}{\mathrm{Ind}}
\nc{\sAn}{\mathrm{can}}
\nc{\Vac}{\mathrm{Vac}}
\nc{\Op}{\mathrm{Op}}
\nc{\Lg}{\check{\fg}}
\nc{\cDelta}{\check{\Delta}}
\nc{\cPhi}{\check{\Phi}}
\nc{\LV}{\check{V}}
\nc{\Lh}{\check{h}}
\nc{\LG}{\check{G}}
\nc{\cT}{\check{T}}
\nc{\ct}{\check{\ft}}
\nc{\cB}{\check{B}}
\nc{\cb}{\check{\fb}}
\nc{\cN}{\check{N}}
\nc{\cn}{\check{\fn}}
\nc{\Spec}{\mathrm{Spec}}
\nc{\End}{\mathrm{End}}
\nc{\crho}{\check{\rho}}
\nc{\clambda}{\check{\lambda}}

\nc{\rX}{\mathring{X}}
\nc{\ru}{\mathring{u}}

\nc{\sW}{\mathscr{W}}
\nc{\sH}{\mathscr{H}}
\nc{\sV}{\mathscr{V}}
\nc{\geom}{\mathrm{geom}}
\nc{\Irr}{\mathrm{Irr}}
\nc{\fm}{\mathfrak{m}}
\nc{\aff}{\mathrm{aff}}
\nc{\Aut}{\mathrm{Aut}}
\nc{\cJ}{\mathcal{J}}
\nc{\fs}{\mathfrak{s}}
\nc{\Stab}{\mathrm{Stab}}
\nc{\st}{\mathrm{st}}
\nc{\tw}{{\widetilde{w}}}
\nc{\gen}{\mathrm{gen}}
\nc{\genn}{\mathrm{genn}}
\nc{\sss}{\mathrm{ss}}
\nc{\fsp}{\mathfrak{sp}}
\nc{\Hom}{\mathrm{Hom}}
\nc{\bm}{\mathbf{m}}
\nc{\HG}{\mathcal{HG}}
\nc{\Gal}{\mathrm{Gal}}
\nc{\Sym}{\mathrm{Sym}}
\nc{\rank}{\mathrm{rank}}

\nc{\tP}{\mathtt{P}}
\nc{\tL}{\mathtt{L}}
\nc{\tU}{\mathtt{U}}

\nc{\tW}{\widetilde{W}}
\nc{\Hk}{\on{Hk}}
\nc{\cL}{\mathcal{L}}
\nc{\talpha}{\widetilde{\alpha}}
\nc{\tQ}{{\widetilde{Q}}}
\nc{\ochi}{\overline{\chi}}
\nc{\tdelta}{\widetilde{\Delta}}
\nc{\wt}{\mathrm{wt}}
\nc{\fQ}{\mathfrak{Q}}

\nc{\Rep}{\mathrm{Rep}}
\nc{\Conn}{\mathrm{Conn}}
\nc{\Hecke}{\mathrm{Hecke}}
\nc{\Gr}{\mathrm{Gr}}
\nc{\GR}{\mathrm{GR}}
\nc{\IC}{\mathrm{IC}}
\nc{\Std}{\mathrm{Std}} 
\nc{\Db}{\mathrm{D}^{\mathrm{b}}}
\nc{\tr}{\mathrm{tr}}
\nc{\gr}{\mathrm{gr}}
\nc{\tmin}{\mathrm{min}}
\nc{\Fun}{\mathrm{Fun}~}

\newcommand{\quash}[1]{}

\AtEndDocument{\bigskip{\footnotesize
		
		\textsc{Tsao-Hsien Chen, School of Mathematics, University of Minnesota, Twin cities, Minneapolis, MN 55455 } \par
		\textit{E-mail address}: \texttt{chenth@umn.edu} \par
		
		\textsc{Lingfei Yi, Shanghai Center for Mathematical Sciences, Fudan University, Shanghai 200438, China} \par
		\textit{E-mail address}: \texttt{yilingfei@fudan.edu.cn} \par
}}

\begin{document} 
	\renewcommand{\thepart}{\Roman{part}}
	
	\renewcommand{\partname}{\hspace*{20mm} Part}
	
	
	\begin{abstract}
		From a stable vector of a stable grading on a simple Lie algebra,
		Yun defined a rigid automorphic datum 
		that encodes an epipelagic representation, 
		and also an irregular connection on the projective line called $\theta$-connection. 
		We show that under geometric Langlands correspondence,
		$\theta$-connection corresponds to the Hecke eigensheaf 
		attached to the rigid automorphic datum,
		assuming the stable grading is inner and its Kac coordinate $s_0$ is positive.  
		We provide applications of the main result 
		on cohomological rigidity of 
		$\theta$-connections, global oper structures, 
		and a de Rham analog of Reeder-Yu's predictions
		on epipelagic Langlands parameters.
	\end{abstract}

	\title{Geometric Langlands for irregular Theta Connections and Epipelagic Representations}
	\author{Tsao-Hsien Chen, Lingfei Yi} 
	\date{\today} 
	\maketitle
	
	\tableofcontents
	
	\section{Introduction}
	Supercuspidal representations are building blocks of the admissible representations of $p$-adic groups.
	Their construction has been a central problem in the area.
	A particular family of supercuspidal representations of depth $\frac{1}{m}$
	for some special numbers $m$
	has been constructed by Reeder-Yu \cite{RY},
	named as \emph{epipelagic representations}.
	The word \emph{epipelagic} indicates that they are among those
	of minimal positive depth.
	In the case of function fields, 
	Yun constructed a family of Hecke eigensheaves
	over a punctured projective line
	that encodes epipelagic representations at a place \cite{YunEpipelagic}.
	In an unpublished work, 
	he conjectured that in the de Rham setting, 
	the Hecke eigenvalues of these eigensheaves should be given by
	the so-called $\theta$-connections, 
	whose properties have been studied by the first author in \cite{Chen}.
	
	In this work we prove Yun's conjecture under some assumptions.
	This also provides a new family of explicit geometric Langlands correspondence
	that allows irregular singularities.
	We first briefly review the above objects in the following.
	\subsection{Epipelagic representations and $\theta$-connections}
	\subsubsection{}
	Let $G$ be a split connected simple and simply connected algebraic group 
	over a local field $K$ with residue field $k$.
	Consider a parahoric subgroup $P\subset G(K)$ and its Moy-Prasad filtration
	$P\supset P(1)\supset P(2)\supset\cdots$.
	The first quotient $L_P:=P/P(1)$ is a reductive group,
	while the other subquotients are abelian and unipotent.
	Let $V_P:=P(1)/P(2)$, on which $L_P$ has an adjoint action.
	If a dual element $\phi\in V_P^*$ has 
	closed orbit and finite stabilizer under the $L_P$ action,
	then we say $\phi$ is \emph{stable}.
	Let $\psi$ be an additive character on $k$,
	then $\psi\circ\phi$ gives a character on $P(1)$ by inflation.
	Reeder and Yu showed that 
	the compactly induced representation $\mathrm{ind}^{G(K)}_{P(1)}(\psi\circ\phi)$
	is a direct sum of irreducible supercuspidal representations 
	\cite[Proposition 5.2]{RY},
	and these irreducible summands are called \emph{epipelagic representations}.
	
	An important observation is that stable vectors in $V_P^*$ are closely related
	to stable gradings of the Lie algebra $\fg/k$ of $G$.
	In this article, we will restrict to the case of inner gradings.
	If the parahoric subgroup $P$ is defined from 
	a rational point $x$ in the Bruhat-Tits building,
	let $m$ be the smallest positive integer such that
	$\clambda:=mx$ is a cocharacter.
	Then $\Ad_{\clambda(\zeta_m)}$ defines an order $m$ grading
	$\fg=\bigoplus_{i\in\bZ_m}\fg_i$.
	Let $G_0\subset G$ be the subgroup with Lie algebra $\fg_0$,
	which acts adjointly on $\fg_i$.
	There exists an isomorphism $(G_0,\fg_1)\simeq(L_P,V_P)$ \cite[Theorem 4.1]{RY}
	that preserves the adjoint actions.
	Therefore, stable elements $\phi\in V_P^*$
	are equivalent to elements in $\fg_1^*\simeq\fg_{-1}$
	that are stable with respect to the action of $G_0$.
	The graded Lie algebras and $G_0$-orbits in $\fg_1$ 
	have been studied by Vinberg \cite{Vinberg}.
	
	\subsubsection{}
	When $K$ has positive characteristic,
	i.e. it is the field of Laurent series over a finite field,
	Yun \cite{YunEpipelagic} constructed a globalization of 
	epipelagic representations
	using the method of \emph{rigid automorphic data}.
	Explicitly, let $X=\bP^1/k$ be the projective line over finite field $k$.
	Let $F=k(t)$ be the function field of $X$, with ring of adeles $\bA_F$.
	Let $P$ be a parahoric subgroup of $G(k(\!(t^{-1})\!))$
	that has a stable functional $\phi\in V_P^*$.
	Let $P^\opp\subset G(k(\!(t)\!))$ be a parahoric subgroup 
	that is opposite to $P$.
	Yun showed that there exists a unique up to scalar 
	nonzero automorphic form on $G(\bA_F)$
	that is unramified over $\bP^1-\{0,\infty\}$,
	fixed by $P^\opp$,
	and is $(P(1),\psi\circ\phi)$-equivariant.
	Such a pair $((P^\opp,\bone),(P(1),\psi\circ\phi))$
	is called an \emph{epipelagic automorphic datum}.
	Moreover, such a form will automatically be a Hecke eigenform.
	
	We can lift this form to a Hecke eigensheaf $\cA_\phi$ 
	on the moduli stack of $G$-bundles over $\bP^1$
	with level structures at $0,\infty$ associated to $P^\opp$ and $P(2)$,
	see \S\ref{ss:epipelagic} for details.
	Then under geometric Langlands correspondence,
	the Hecke eigenvalues of $\cA_\phi$
	give rise to a $\LG$-local system $\cE_\phi$ on $\bP^1-\{0,\infty\}$,
	where $\LG$ is the dual group of $G$.
	In particular, when $P=I$ is the Iwahori subgroup,
	the above construction recovers the construction in \cite{HNY}.
	
	\subsubsection{}
	When $k=\bC$, we have a de Rham analog of the above global construction,
	replacing all the sheaves with $D$-modules.
	Then the Hecke eigen $D$-module $\cA_\phi$
	gives rise to a de Rham $\LG$-local system $\cE_\phi$,
	i.e. a meromorphic connection 
	on the trivial principal $\LG$-bundle over $\bP^1$
	with singularities at $0,\infty$.
	When $P=I$,
	Zhu \cite{Zhu} proved that $\cE_\phi$ coincides with 
	the irregular connection constructed by Frenkel and Gross in \cite{FGr},
	which is now called the \emph{Frenkel-Gross connection}.
	
	\subsubsection{}
	In an unpublished work, Yun constructed 
	a generalization of Frenkel-Gross connections
	from a stable grading of the Lie algebra $\Lg$ of $\LG$
	and a stable vector $X$ in the degree one subspace $\Lg_1$,
	which is called a \emph{$\theta$-connection} $\nabla^X$, 
	cf. \cite{Chen} or \S\ref{ss:theta conn} below.
	In \cite{Chen}, Chen proved under a condition on the stable grading
	that these connections are \emph{cohomologically rigid} \cite[Theorem 5.2]{Chen}.
	The constraint on the stable grading is that 
	the associated Kac coordinate $s_0$ is nonzero.

	\subsection{The geometric Langlands correspondence between $\theta$-connections and epipelagic automorphic datum}
	There exists a bijection 
	between the stable gradings of $\fg$ and its dual Lie algebra $\Lg$
	such that the stable orbits in $V_P^*$ and $\Lg_1$
	can be matched bijectively (see Lemma \ref{l:GIT isom}):
	\begin{equation}\label{bijection}
		V_P^{*,\st}/L_P\simeq\Lg_1^\st/\LG_0.
	\end{equation} 
	Then it is conjectured in \cite[Conjecture 1.1]{Chen},
	originally by Yun, 	
	that in the de Rham setting,
	the eigenvalue $\cE_\phi$ coming from an epipelagic automorphic datum
	is the same as the $\theta$-connection defined from the corresponding
	stable grading and stable orbit.
	When $P=I$, this is the main result of \cite{Zhu}.
	
	Our main result confirms the conjectural correspondence  between
	epipelagic automorphic data and $\theta$-connections
	for a large family of stable gradings:
	\begin{thm}(Theorem \ref{t:main})\label{t: main thm intro}
		Assume both stable inner gradings of order $m$ on $\Lg$ and $\fg$
		have Kac coordinate $s_0>0$.
		For a stable vector $X$ and the corresponding stable functional $\phi$,
		$\nabla^X$ is isomorphic to the Hecke eigenvalue $\cE_\phi$.
	\end{thm}
	
	A stable inner grading is uniquely determined by the order $m$.
	According to the tables in \cite[\S7]{RLYG} classifying stable gradings, 
	the order $m$ of the stable inner gradings on $\fg$
	considered in Theorem \ref{t: main thm intro} are as follows:
	\begin{equation}\label{eq:inner types both s_0>0}
		\begin{cases}
			A_n,\qquad m=n+1,\\
			B_n,\qquad m=n\ \text{even or}\ m=2n,\\
			C_n,\qquad m=n\ \text{even or}\ m=2n,\\
			D_n,\qquad m=n\ \text{even or}\ m=2n-2,\\
			E_6,\qquad m=6,9,12,\\
			E_7,\qquad m=6,14,18,\\
			E_8,\qquad m=6,10,12,15,20,24,30,\\
			F_4,\qquad m=4,6,8,12,\\
			G_2,\qquad m=3,6.
		\end{cases}
	\end{equation}

	\begin{rem}
	\begin{itemize}
    	\item [(i)]
		When $m=h$ is the Coxeter number of $G$, 
		where $P=I$ is the Iwahori subgroup, 
		$\nabla^X$ recovers the Frenkel-Gross connection.
		In this case the theorem is known in \cite{Zhu}.

		\item [(ii)]
	    We expect that Theorem \ref{t: main thm intro} is true for all
	    stable inner gradings on $\Lg$ with $s_0>0$, 
	    where the stable inner grading on $\fg$ of the same order
	    may have $s_0=0$. 
	    According to the tables in \cite[\S7]{RLYG}
	    the only missing case is $(\fg=\so_{2n+1},\Lg=\fsp_{2n})$.
	    We provide evidence in Lemma \ref{l:Lg=sp X_1 in u_P}.

		\item [(iii)]		
		In the case of stable gradings
	    on $\Lg$ with  $s_0=0$, 
	    we discover that the $\theta$-connection $\nabla^X$ might not 
	    have the same monodromy as that of $\mathcal E_\phi$ at $0$, 
	    see \S\ref{sss:tame monodromy counterexamples}. 
	    It turns out that our proof of Theorem \ref{t: main thm intro}
	    also works when $\nabla^X$ does have the same local monodromy 
	    as $\cE_\phi$ at $0$.
	    In \S\ref{s:conj for s_0=0}, we propose a version of
	    geometric Langlands correspondence  for epipelagic automorphic data and $\theta$-connections with $s_0=0$
	    and provide an example when $\Lg=\so_7$
	    in \S\ref{ss:eg of min orbit conj}.
	    
	    \item [(iv)]
	    By Corollary \ref{c:theta conn 0 monodromy}.(i),
	    the $\theta$-connections considered in the above theorem
	    belong to \emph{generalized Kloosterman connections}
	    studied in the work of Hohl and Jakob \cite{HJrigid}.
	    They prove the physical rigidity of generalized Kloosterman connections
	    in the Theorem 1.2.1 of \emph{loc. cit.} using Stokes data.
	\end{itemize}
	\end{rem}
	
	In Remark \ref{rem:s_0}, we will have a more detailed discussion 
	on the necessity of imposing the assumptions on $s_0$
	and situations where they can be relaxed.

	\subsection{Idea of proof}\label{ss:idea of proof}
	Theorem \ref{t: main thm intro} generalizes 
	Zhu's result for Frenkel-Gross connections.
	However, many arguments in \cite{Zhu} 
	cannot be generalized to $\theta$-connections.
	The issue is that the proof in \emph{loc. cit.} made use of 
	the obvious global oper structure of Frenkel-Gross connection
	and the flatness of the relevant Hitchin map.
	Here the oper structure on a connection is what makes it possible
	to apply Beilinson-Drinfeld's construction of the Hecke eigensheaf
	of the connection and the flatness would ensure that the Hecke eigensheaf is non-zero.
	
	Roughly speaking, 
	to prove the correspondence between $\nabla^X$ 
	and $\cE_\phi$ in Theorem \ref{t: main thm intro},
	we instead begin with a global oper on $\bGm$
	that has the same local monodromy at $\infty$ as $\nabla^X$
	and has a particular form of equation at $0$.
	We construct a non-zero Hecke eigensheaf of it that is 
	$(P(1),\cL_\phi)$-equivariant.
	By the rigidity of the epipelagic automorphic datum, 
	this must be Hecke-eigen with the same eigenvalue as $\cA_\phi$.
	Thus the underlying connection of the oper 
	must coincide with the Hecke eigenvalue $\cE_\phi$.
	Lastly, such oper will automatically have the same local monodromy as $\nabla^X$.
	It follows from the physical rigidity theorem of Jakob-Hohl that 
	$\nabla^X$ is the eigenvalue of $\cA_\phi$.
	
	A main difficulty is that the space of global opers we need is complicated.
	In fact, it is even non-obvious whether there exists such global opers.
	Besides, the Hitchin map in our general setting is not always flat 
	because the moduli stack of $G$-bundles with corresponding level structures
	is not a \emph{good stack}, c.f. \S\ref{ss:moment not flat}.
	Thus we cannot construct the Hecke eigensheaf exactly as in \cite{Zhu}.
	Also, in our situation the local monodromies of $\theta$-connections
	are more complicated than Frenkel-Gross connections.
	To overcome these difficulties, 
	we compute central support of parahoric induced representations of
	affine Kac-Moody algebras, 
	compute the parahoric local Hitchin image,
	compare the local monodromy of $\theta$-connections at $0$
	with the nilpotent conjugacy class associated to
	the parahoric subgroup $P$ in \S\ref{s:mon at 0},
	and study their irregular singularity at $\infty$ in \S\ref{s:theta conn local monodromy at infty}.

	\subsection{Applications}
	From Theorem \ref{t: main thm intro},
	we can deduce corollaries on the epipelagic Langlands parameters.
	In \cite[\S7.1]{RY}, Reeder and Yu made predictions 
	on the Langlands parameters of epipelagic representations.
	In \S\ref{s:applications on epipelagic L param} we confirm the de Rham analog 
	of their predictions as a corollary of our global result.
	We also deduce the de Rham version of some results in \cite{FuGuEuler}
	on Euler characteristics of generalized Kloosterman sheaves 
	for classical groups with respect to the standard representation.
	In a recent work of the second author with Daxin Xu \cite{XuYi},
	the above results are used to study the 
	positive characteristic local and global epipelagic Langlands parameters.

	\subsection{Future directions}
	\subsubsection{}
	The construction of epipelagic representations in \cite{RY} 
	does not require the group $G$ to be split and simply connected:
	they only need $G$ to be a connected reductive group
	that splits over a tamely ramified extension of the local field.
	Also, in \cite{YunEpipelagic} $G$ is just a quasi-split reductive group.
	Besides, $\theta$-connections can be constructed from stable gradings
	that are not necessarily inner.
	Thus the correspondence between epipelagic automorphic data
	and $\theta$-connections can be naturally extended to quasi-split groups
	and general stable gradings.
	However, the notion of opers, 
	the Feigin-Frenkel isomorphism,
	and the construction of Hecke eigensheaves 
	using Beilinson-Drinfeld's localization method
	have not been developed for quasi-split groups and twisted Lie algebras.
	Thus our method cannot yet be applied to the twisted case.
	But it would be interesting to explore whether the twisted correspondence 
	can be associated to the untwisted case
	by going to a covering of the curve. 
	
	\subsubsection{}
	Another assumption we made is that we require the stable inner grading 
	has Kac coordinate $s_0>0$.
	On the automorphic side 
	this corresponds to that the parahoric subgroup $P$ 
	is contained in the hyperspecial subgroup $G(\cO)$.
	One reason is that the construction of $\theta$-connections 
	needs some modifications:
	when $s_0=0$, there are some stable vectors $X$ such that
	the monodromy at $0$ of the associated $\theta$-connection $\nabla^X$
	does not match with that of the eigenvalue $\cE_\phi$ 
	of the epipelagic Hecke eigensheaf, i.e. the class $\uu_P$.
	Therefore they cannot be isomorphic.
	It is to be explored what should be the correct statement for $s_0=0$,
	on which we provide some discussion in \S\ref{s:conj for s_0=0}.
	
	\subsubsection{}
	In \cite{KSRigid}, Kamgarpour and Sage constructed a collection of
	cohomological rigid connections that generalize Frenkel-Gross connection,
	but do not cover other $\theta$-connections.
	One family in their construction generalizes Airy equation from
	general linear groups to simple algebraic groups.
	Meanwhile, a family of rigid automorphic data has been constructed
	in \cite{JKY} for any simple reductive group
	that corresponds to Airy local systems for general linear groups.
	In \cite{YiToral}, the second author proves that 
	the above two families match under geometric Langlands.
	The strategy is similar as the current paper, 
	based on a newly discovered local geometric Langlands correspondence
	between irreducible isoclinic connections
	and toral supercuspidal representations.
	
	More generally, all the aforementioned connections belong to the family
	considered in \cite{JYDeligneSimpson}.
	We hope to extend our method to this generality in the future.

	\subsubsection{}\label{sss:pos applications}
	The proof in \cite{XuZhu} is applicable in our case,
	giving a Frobenius structure on $p$-adic $\theta$-connections,
	whose eigenvalues would correspond to the $\ell$-adic epipelagic eigenvalues.
	In a joint work of the second author with Daxin Xu \cite{XuYi},
	based on Theorem \ref{t: main thm intro},
	they prove that the Langlands parameters of epipelagic representations
	have the properties predicted by Reeder-Yu \cite{RY},
	which implies the cohomological rigidity of
	the associated $\ell$-adic $\LG$-local systems,
	see \cite[Theorem 1.3.2]{XuYi}.
	They also show that the corresponding 
	overconvergent $F$-isocrystals for $\LG$ are physically rigid
	\cite[Theorem 1.3.9]{XuYi}.

	\subsection{Organization of the article}
	In \S\ref{s:set up},
	we set up notations, 
	take a more detailed review of our main players: 
	epipelagic automorphic data and $\theta$-connections,
	and give a precise statement of the Langlands correspondence between them,
	see Lemma \ref{l:GIT isom} and Theorem \ref{t:main}.
	In \S3 we develop some properties about stable vectors in a graded Lie algebra
	needed in the proof of the main theorem, 
	which are of independent interest.
	In \S\ref{s:theta conn local monodromy at infty} 
	and \S\ref{s:local opers with slope 1/m},
	we first analyze the local monodromies of $\theta$-connections
	at the irregular singularity,
	then we discuss their possible local oper structures.
	In \S\ref{s:global opers},
	we compute the space of global opers we need to apply
	the localization method.
	In \S\ref{s:proof of main thm}, we prove the main result Theorem \ref{t:main}.
	In \S\ref{s:applications on epipelagic L param}, 
	we give applications on the de Rham Langlands parameters.
	In \S\ref{s:conj for s_0=0}, we discuss a potential correction to the
	Theorem \ref{t:main} in the $s_0=0$ case.

	\subsection{Acknowledgment}
	The authors thank Gurbir Dhillon,
	Joakim Færgeman, Sam Raskin, and Daxin Xu
	for useful discussions. 
	T.-H. Chen also thanks the NCTS-National Center for Theoretical Sciences at Taipei where parts of this work were done. 
	The research of T.-H. Chen is
	supported by NSF grant DMS-2143722.
	L. Yi is supported by Start-Up Grant No. JIH1414042Y of Fudan University.

	\section{Set up}\label{s:set up}
	Let $G$ be a simple, simply-connected complex algebraic group. 
	Denote the Lie algebra of $G$ by $\fg$. 
	Let $\bP^1/\bC$ be the projective line, 
	$t$ a coordinate at $0$, 
	$s=t^{-1}$ a coordinate at $\infty$,
	function field $F=\bC(t)$. 
	Denote the formal disk and punctured formal disk at $x\in\bP^1$ by
	$D_x,D_x^\times$,
	denote the local field and ring of integers at $x$ by $F_x,\cO_x$. 
	Let $\LG$ be the dual group of $G$, with Lie algebra $\Lg$.
	
	\subsection{Epipelagic rigid automorphic data}\label{ss:epipelagic}
	Consider a standard parahoric subgroup $P$ of $G(\!(s)\!)$. 
	Let $P(1),P(2)$ be the first two Moy-Prasad subgroups of $P$, 
	and let $L_P\subset P$ be the natural lift of the Levi quotient $L_P\simeq P/P(1)$. 
	Denote $V_P=P(1)/P(2)$. 
	A parahoric subgroup $P$ is called \emph{admissible} 
	if $V_P^*$ contains a nonzero element 
	with closed $L_P$-orbit and finite stabilizer in $L_P$.
	Such a parahoric subgroup uniquely corresponds to 
	a regular elliptic number $m$, c.f. \cite[\S2.6]{YunEpipelagic}.  
	Let $V_P^{*,\st}$ be the locus of stable linear forms, 
	and take arbitrary $\phi\in V_P^{*,\st}$. 
	Let $P^\opp\subset G(\!(t)\!)$ be the parahoric subgroup opposite to $P$,
	$\bone$ be the trivial connection on $P^\opp$.
	Denote by $\cL_\phi$ the pullback of the exponential $D$-module to $P(1)$
	via quotient $P(1)\rightarrow P(1)/P(2)$ and $\phi$.
	We call the pair $(P^\opp,\bone;P(1),\cL_\phi)$
	an \emph{epipelagic automorphic datum}.
	
	Let $\cG$ be the group scheme on $\bP^1$ such that 
	$\cG|_{D_0}=P^\opp$, $\cG|_{D_\infty}=P(2)$, $\cG|_{\bGm}=G\times\bGm$.
	Let $\Bun_\cG$ be the moduli stack of $\cG$-bundles on $\bP^1$.
	In \cite[Proposition 2.11, \S 3]{YunEpipelagic}, 
	Yun showed that any epipelagic automorphic datum is \emph{rigid},
	i.e. there exists a unique irreducible holonomic $D$-module $\cA_\phi$ on $\Bun_{\cG}$ that is $(P(1),\cL_\phi)$-equivariant 
	withe respect to the action of $P(1)$ on $\Bun_{\cG}$ at $\infty$.
	Explicitly, let $j:V_P=P(1)/P(2)\hookrightarrow\Bun_{\cG}$
	be the embedding given by the action of $P(1)$ on the trivial $\cG$-bundle.
	By rigidity we have clean extension
	$\cA_\phi=j_!\cL_\phi=j_*\cL_\phi$.
	In the \emph{loc. cit.}, Yun showed that $\cA_\phi$ is a Hecke eigensheaf. 
	We denote the eigenvalue of $\cA_\phi$ by $\cE_\phi$,
	which is a $\LG$-connection on $\bGm$.

	\subsection{$\theta$-connections}\label{ss:theta conn}
	On the other hand, let $\theta$ be a stable inner grading of $\Lg$ in the unique class corresponding to the regular elliptic number $m$
	via \cite[Corollary 14]{RLYG}.
	Let $\Lg_1\subset\Lg$ be the degree one subspace with respect to $\theta$, and $\LG_0\subset\LG$ the fixed reductive subgroup. 
	Let $\Lg_1^{\st}$ be the locus of stable vectors, i.e. those whose $\LG_0$-conjugacy class is closed with finite centralizer. 
	Take arbitrary $X\in\Lg_1^{\st}$. 
	In \cite{Chen}, following an unpublished work of Yun, 
	Chen defined a $\LG$-connection $\nabla^X$ on $\bGm$ from $X$ called a \emph{$\theta$-connection}.
	When the particular Kac coordinate $s_0$ of $\theta$ is nonzero, 
	Chen proved that $\nabla^X$ is cohomologically rigid \cite[Theorem 5.2]{Chen}.
	
	Precisely, assume $\theta$ is a torsion inner automorphism given by the adjoint action of $\exp(x)$, 
	$x\in\mathbb{X}_*(\cT)\otimes_{\bZ}\bQ$
	is a barycenter of a facet, 
	such that $\clambda=mx\in\mathbb{X}_*(\cT)$ 
	where $m\in\bZ_{>0}$ is minimal. 
	The point $x$ can be conjugated by affine Weyl group into the closure of fundamental alcove, 
	thus we will assume $x$ is in this closure. 
	Denote the grading defined by $\clambda$ on $\Lg$ by
	\begin{equation}\label{eq:Z-grading}
		\Lg=\bigoplus_k\Lg(k).
	\end{equation}
	This gives a grading on $\Lg_1$:
	\begin{equation}
		\Lg_1=\bigoplus_{k\equiv 1\!\!\!\mod m}\Lg(k).
	\end{equation}
	Write $X$ as $X=\sum_k X_k$, $X_k\in\Lg_1(k)$.
	Let $\beta$ be the highest root, 
	$\alpha_0=1-\beta$,
	$s_0=m\alpha_0(x)$. 
	Here the Kac coordinate $s_0=0$ or $1$.
	By \cite[Lemma 2.1]{Chen}, 
	those $k$ with $X_k\neq0$ satisfy $s_0-m\leq k\leq 1$ and $m|k-1$.
	Thus indeed $X=X_1+X_{1-m}$. 
	The following is called the \emph{$\theta$-connection} attached to $X$:
	\begin{equation}\label{eq:theta conn}
		\nabla^X:=\td+\sum_k X_k t^{\frac{1-k}{m}}\frac{\td t}{t}
		=d+(X_1+X_{1-m}t)\frac{dt}{t}.
	\end{equation}

	\subsection{The Langlands correspondence}
	\begin{lem}\label{l:GIT isom}
		There exists a natural isomorphism
		\begin{equation}\label{eq:GIT isom}
			V_P^*/\!/L_P\simeq\Lg_1/\!/\LG_0
		\end{equation}
		that restricts to a bijection between stable conjugacy classes:
		\begin{equation}\label{eq:st bij}
			V_P^{*,\st}/L_P\simeq\Lg_1^\st/\LG_0
		\end{equation}
	\end{lem}
	\begin{proof}
		Let $(G_0, \fg_1)$ be the Vinberg's pair for grading on $\fg$ associated to the regular elliptic number $m$. 
		Then we have isomorphism $(G_0,\fg_1)\simeq(L_P,V_P)$ that preserves the adjoint action. 
		Denote $\fc^*=\fg^*/\!/G\simeq\ft^*/\!/W$ 
		and $\cfc=\Lg/\!/\LG\simeq\ct/\!/W$. 
		The equality $\ft^*=\ct$ induces an isomorphism $\fc^*\simeq\cfc$.
		
		Observe that the projection $\fg=\bigoplus_i\fg_i\rightarrow\fg_1$
		induces the dual map $\fg_1^*\rightarrow\fg^*$.
		The composition $\fg_1^*\rightarrow\fg^*\rightarrow\fg^*/\!\!/G=\fc^*$
		factors through $\fg_1^*/\!\!/G_0\rightarrow\fc^*$.
		On the other hand, map $\Lg_1\rightarrow\Lg\rightarrow\cfc$
		factors through $\Lg_1/\!/\LG_0\rightarrow\cfc$.
		Fix a set of homogeneous generators $\{F_i\}$ of $\bC[\fg^*]^G$, 
		which also gives a set of homogeneous generators 
		of $\bC[\Lg]^{\LG}$ via $\fc^*\simeq\cfc$. 
		From \cite[Theorem 3.5.(i)]{Panyushev}, 
		we know the above two maps are closed embeddings:
		\[
		\fg_1^*/\!/G_0\hookrightarrow\fc^*,\qquad \Lg_1/\!/\LG_0\hookrightarrow\cfc,
		\]
		where the images are both equal to the closed subspace 
		defined by the vanishing of all the generators $F_i$ 
		whose degree is not divisible by $m$. 
		Here we use that $\fg_1^*\simeq\fg_{-1}$ is just the degree one eigenspace of the inverse grading $\theta^{-1}$. 
		This induces the desired isomorphism
		\[
		V_P^*/\!/L_P\simeq\fg_1^*/\!/G_0\simeq\Lg_1/\!/\LG_0.
		\]
		
		The above isomorphism \eqref{eq:GIT isom} induces a bijection 
		between semisimple conjugacy classes in $V_P^*$ and $\Lg_1$. 
		By \cite[Lemma 13]{RLYG}, 
		it induces a bijection between stable classes.
	\end{proof}
	
	\begin{rem}
		Such an isomorphism for $P=I$ is used in \cite[Theorem 4.3.3]{XuZhu}.
	\end{rem}
	
	\begin{thm}\label{t:main}\
		Assume both stable inner gradings of order $m$ on $\Lg$ and $\fg$
		have Kac coordinate $s_0>0$.
		Let $\phi\in V_P^{*,\st}$ be a stable functional, 
			$X=X(\phi)$ be a representative of the stable vector orbit 
			given by the image of $\phi$ under \eqref{eq:GIT isom}. 
			There is an isomorphism of $\LG$-connections: 
			\begin{equation}\label{eq:main isom}
				\cE_\phi\simeq\nabla^X.
			\end{equation}
	\end{thm}

    We will prove the theorem 
    following the strategy described in \S\ref{ss:idea of proof}.
    
    Henceforth, unless otherwise stated, 
    we assume $\theta$ satisfies the assumptions in the above theorem.
    We summarize below how the assumptions can be relaxed 
    in different sections of the article.
	
	\begin{rem}\mbox{}\label{rem:s_0}
		\begin{itemize}
			\item [(i)] 
			The condition $s_0>0$ for the grading on $\fg$ is used only in
			the computation of central supports of parahoric vacuum modules
			\S\ref{s:global opers}
			and parahoric local Hitchin images \S\ref{s:appendeix}
			to overcome a technical difficulty.
            Therefore:
            \begin{itemize}
            	\item 
            	In \S\ref{s:mon at 0} we only need to assume $\theta$ is a stable inner
            	grading on $\fg$ whose Kac coordinate $s_0>0$.
            	\item 
            	In \S\ref{s:theta conn local monodromy at infty} 
            	and \S\ref{s:local opers with slope 1/m},
            	$\theta$ can be any stable inner grading.
            	\item 
            	We will see that the proof of the theorem applies as long as
            	$s_0>0$ for the stable inner grading on $\fg$
            	and the monodromy at $0$ of $\nabla^X$ 
            	is contained in the closure of that of $\cE_\phi$.
            	\item 
            	In Appendix \ref{s:appendeix} we do not require any assumption 
            	on $s_0$.
            \end{itemize}
			
			\item [(ii)]
			The conditions $s_0>0$ for the stable inner gradings 
			of the same order on $\fg$ and $\Lg$
			are consistent as long as the root system is self dual,
			i.e. except in type $B_n$, $C_n$.
			According to the tables in \cite[\S7]{RLYG},
			stable gradings on $\fsp_{2n}$ all satisfy $s_0>0$,
			while on $\so_{2n+1}$ there are both stable gradings with
			$s_0=0$ or $s_0>0$.
			In view of the above discussion, there are four situations:
			\begin{itemize}
				\item [(1)] The gradings on both $\fg$ and $\Lg$ satisfy $s_0>0$.
				The theorem holds in this case.
				
				\item [(2)] The gradings on both $\fg$ and $\Lg$ satisfy $s_0=0$.
				This can happen only away from type $B_n$, $C_n$.
				In this case, 
				we discover that $\nabla^X$ might not have the same monodromy as
				$\mathcal E_\phi$ at $0$, 
				see \S\ref{sss:tame monodromy counterexamples}. 
				In \S\ref{s:conj for s_0=0}, we propose a version of
				geometric Langlands correspondence for epipelagic automorphic data
				and $\theta$-connections allowing $s_0=0$
				and provide an example when $\Lg=\so_7$
				in \S\ref{ss:eg of min orbit conj}.
				
				\item [(3)] On $\fg$, $s_0>0$, while on $\Lg$, $s_0=0$.
				This happens only if $\fg=\fsp_{2n}$ 
				and $\Lg=\so_{2n+1}$.
				In this case, the monodromies of $\nabla^X$ and $\cE_\phi$ at $0$ can either match or not.
				See \S\ref{ss:s_0=0 ex with corr mon} 
				for examples of $\nabla^X$ with matching monodromy at $0$,
				and \S\ref{ss:eg of min orbit conj} for $\nabla^X$ 
				whose monodromy does not match.
				
				\item [(4)] On $\fg$, $s_0=0$, while on $\Lg$, $s_0>0$.
				This happens only if $\fg=\so_{2n+1}$ and $\Lg=\fsp_{2n}$.
				From Lemma \ref{l:Lg=sp X_1 in u_P}, we will see
				the monodromies at $0$ always match in this case.
				We expect the theorem holds as well in this situation,
				as the current use of condition $s_0>0$ on $\fg$
				is only to tackle a technical difficulty.			
			\end{itemize}
		\end{itemize}
	\end{rem}

	\section{Local monodromy at $0$}\label{s:mon at 0}
	\subsection{Matching monodromies at $0$}
	One ingredient of the proof of Theorem \ref{t:main} is that
	the local monodromy of $\nabla^X$ at $0$ coincides with 
	the local monodromy of $\cE_\phi$ at $0$.
	
	For $\nabla^X$, it is clear from \eqref{eq:theta conn} that
	its local monodromy at $0$ is just $\exp(X_1)$.
	
	For $\cE_\phi$, by \cite[Theorem 4.5]{YunEpipelagic} 
	its local monodromy at $0$ is the class $\uu_P$
	given by Lusztig's bijection 
	between two-sided cells of affine Weyl group of $G$
	and unipotent conjugacy classes of $\LG$.
	Under our assumption $s_0>0$, 
	$\uu_P$ can be described as follows.
	In this case $P\subset G(\cO)$ is the preimage of a 
	standard parabolic subgroup $P_0\subset G$ under reduction map.
	Let $\hat{P_0}\subset\LG$ be the dual standard parabolic subgroup.
	Then by \cite[Proposition 4.8.(2)]{YunEpipelagic}, 
	$\uu_P$ is the Richardson class of $\hat{P_0}$.
	Denote by $U_{\hat{P_0}}$ the unipotent radical of $\hat{P_0}$
	with Lie algebra $\fu_{\hat{P_0}}$,
	then the closure of $\uu_P$ is $\Ad_{\LG}\fu_{\hat{P_0}}$.
	
	\subsubsection{}
	We first establish the following weaker result.
	\begin{prop}\label{p:0 mon match}
		Any stable vector $X=X_1+X_{1-m}$ satisfies
		$X_1\in\Ad_{\LG}\fu_{\hat{P_0}}$.
	\end{prop}
	\begin{proof}
		First consider when the root system of $\fg$ is self-dual,
		which is the case if $G$ is not of type $B_n$, $C_n$.
		Recall the $\bZ$-grading \eqref{eq:Z-grading} on $\Lg$ defined by $\clambda=mx$ refines the stable grading.
		Then $\clambda$ defines a standard parabolic subgroup $\hat{P}\subset\LG$ with unipotent radical $U_{\hat{P}}$.
		Their Lie algebras are given by
		\[
		\hat{\fp}=\bigoplus_{k\geq 0}\Lg(k),
		\qquad \fu_{\hat{P}}=\bigoplus_{k\geq1}\Lg(k).
		\]
		
		Since the root system is self-dual and $s_0>0$,
		we have $\hat{P_0}=\hat{P}$.
		Thus $X_1\in\Lg(1)\subset\fu_{\hat{P}}=\fu_{\hat{P_0}}$.
		
		When $\fg$ is of type $B_n$ or $C_n$, 
		if $m=2n$ is the Coxeter number, $\hat{P_0}$ is a Borel subgroup,
		where the statement is trivial.
		Otherwise, there are only two cases
		where both stable gradings on $\fg$ and $\Lg$ of the same order
		satisfy $s_0>0$:
		when $\Lg=\so_{2n+1}$, $\fg=\fsp_{2n}$, $m=n$ is even;
		or when $\Lg=\fsp_{2n}$, $\fg=\so_{2n+1}$, $m=n$ is even.
		We will verify by hand in Lemma \ref{l:Lg=sp X_1 in u_P}
		and Lemma \ref{l:Lg=so_2n+1 X_1 in u_P} that 
		$X_1$ always has the same Jordan type as 
		the Richardson class $\uu_P$.
	\end{proof}
	
	We will use the following fact in the computations 
	in \S\ref{ss:0 mon Lg=sp} and \S\ref{ss:0 mon Lg=so_2n+1 m=n even}.
	\begin{lem}\label{l:st=rs}
		For a stable grading $\theta$ on $\Lg=\bigoplus_i\Lg_i$, 
		a vector $X\in\Lg_1$ is stable if and only if 
		it is regular semisimple.
	\end{lem}
	\begin{proof}
		By \cite[Lemma 13]{RLYG}, we only need to show that
		for $X\in\Lg_1$ regular semisimple, 
		$\theta$ has no fixed point on the Cartan subalgebra $\fz(X)$.
		Let $\fc\subset\Lg_1$ be the Cartan subspace containing $X$.
		Then $\fz_{\Lg}(\fc)=\fz(X)$ is a $\theta$-stable Cartan subalgebra.
		Since the grading is stable, there exists stable vector $Y\in\Lg_1$.
		Since all the Cartan subspaces of $\Lg_1$ are all conjugated by $\LG_0$,
		there exists $g\in\LG_0$ such that $\Ad_g(Y)\in\fc$.
		Thus $\fz(X)=\Ad_g(\fz(Y))$.
		By \cite[Lemma 13]{RLYG}, $\fz(Y)^\theta=0$.
		As $\theta(g)=g$, we get
		$\fz(X)^\theta=\Ad_g(\fz(Y)^\theta)=0$.
		This shows that $X$ is stable.
	\end{proof}
    
    \subsubsection{}
    Denote by $u_{\hat{P}_0}$ the unique open dense
    $\hat{P}_0$-orbit in $\fu_{\hat{P}_0}$ where
    $\exp(u_{\hat{P}_0})=\uu_P\cap U_{\hat{P}_0}$. 
    Let $\cO_\theta$ be the unique open dense $\LG(0)$-orbit in $\Lg(1)$.
	\begin{cor}\label{c:theta conn 0 monodromy}
		Assumptions as before.
		\begin{itemize}
			\item [(i)]
			Any stable vector $X=X_1+X_{1-m}\in\Lg_1=\Lg(1)\oplus\Lg(1-m)$ satisfies
			$\exp(X_1)\in\uu_P$.
			\item [(ii)]
			We have $\cO_\theta=u_{\hat{P}_0}\cap\Lg(1)$.
			As a result, any stable vector $X=X_1+X_{1-m}$
			satisfies $X_1\in\cO_\theta$.
		\end{itemize}
	\end{cor}
	\begin{proof}
		(i):
		On the one hand, by \cite[Theorem 5.1, Theorem 5.2]{Chen}\footnote{In \cite[Theorem 5.1]{Chen}, the vector $X$ should be assumed to be not only regular semisimple, but stable. Otherwise, the Lemma 6.1 of \emph{loc. cit.} fails for $n=0$, which is because the pairing $\langle t\partial_t(z),z'\rangle_\sigma$ in the proof of Proposition 3.5 (3) of \emph{loc. cit.} always vanishes for $z,z'\in\mathfrak{a}_0$. When $X$ is stable, $\mathfrak{a}_0=0$, so that Lemma 6.1 and Theorem 5.1 hold under the stable assumption.}, we have
		$\dim\Lg^{X_1}=\frac{|\Phi|}{m}$.
		On the other hand, by the proof of \cite[Proposition 5.2]{YunEpipelagic},
		we have that for $u\in u_{\hat{P}_0}$, $\dim\Lg^u=\frac{|\Phi|}{m}$.
		Thus $\dim\Lg^{X_1}=\dim\Lg^u$, 
		and the $\LG$-orbits of $X_1$ and $u\in\uu_P$ have the same dimension.
		By Proposition \ref{p:0 mon match}, 
		$X_1$ is in the closure of $u_{\hat{P}_0}$.
		Thus we must have $X_1\in u_{\hat{P}_0}$,
		$\exp(X_1)\in\uu_P$.
		
		(ii):
		When $G$ is of type $B_n$ or $C_n$,
		the proposition will be verified in 
		Lemma \ref{l:Lg=sp X_1 in u_P} and Lemma \ref{l:Lg=so_2n+1 X_1 in u_P}.
		
		When $G$ is of other types, the dual parabolic $\hat{P_0}$ coincides
		with the parabolic $\hat{P}$ defined from $\clambda$.
		Denote by $L_{\hat{P}}, U_{\hat{P}}$ the 
		standard Levi subgroup and unipotent radical of $\hat{P}$,
		with Lie algebra $\fu_{\hat{P}}=\Lie(U_{\hat{P}})$.
		Denote by $u_{\hat{P}}$ the unique open dense
		$\hat{P}$-orbit in $\fu_{\hat{P}}$ where
		$\exp(u_{\hat{P}})=\uu_P\cap U_{\hat{P}}$. 
		Observe that $L_{\hat{P}}=\LG(0)=\LG_0$
		and $\fu_{\hat{P}}=\bigoplus_{k\geq1}\Lg(k)$.
		
		By Lemma \ref{l:st=rs}, 
		stable vectors are open dense in $\Lg_1$,
		so that their projection in $\Lg(1)$ is open dense.
		Thus there exists stable $X=X_1+X_{1-m}$ such that $X_1\in\cO_\theta$.
		By part (i), $X_1\in u_{\hat{P}_0}$. 
		Thus $\cO_\theta\subset u_{\hat{P}}\cap\Lg(1)$.
		On the other hand, the projection of 
		$u_{\hat{P}}
		=\hat{P}\cdot X_1
		=U_{\hat{P}}\cdot(\LG(0)\cdot X_1)\subset\fu_{\hat{P}}$ 
		to $\Lg(1)$ is $\LG(0)\cdot X_1$, 
		a single $\LG(0)$-orbit.
		So that the projection of $u_{\hat{P}}\cap\Lg(1)$ to itself 
		is a single $\LG(0)$-orbit.
		We conclude $\LG(0)\cdot X_1=u_{\hat{P}}\cap\Lg(1)=\cO_\theta$.
	\end{proof}

	\begin{rem}\mbox{}
		\begin{itemize}
			\item [(1)]
			When $\LG$ is of classical type, 
			i.e. $A_n, B_n, C_n, D_n$,
			the regular semisimple elements
			are those matrices whose eigenvalues are nonzero under roots.
			Then the proposition can be verified directly case-by-case.
			
			\item [(2)]
			When $\LG$ is of exceptional type, 
			i.e. $E_6, E_7, E_8, F_4, G_2$,
			there is a direct argument for part (ii). 
			By \cite[Proposition 26]{RLYG},
			we know in exceptional types
			there exists a distinguished nilpotent element $A\in\fg(1)$
			such that $\ad(A):\Lg(0)\rightarrow\Lg(1)$ is a bijection,
			and there exists vector $M\in\Lg(1-m)$ such that 
			$M+A\in\Lg_1^\st$.
			From this we know $\LG(0)\cdot A$ is dense in $\Lg(1)$,
			thus $\LG(0)\cdot A=\cO_\theta$.
			Also, we know from \cite[Proposition 26, Remark 2]{RLYG} 
			that $\cO_\theta$ corresponds to the regular elliptic class representing $\theta$ in the Weyl group via the Kazhdan-Lusztig map, 
			and from the last sentence of \cite[\S1.2]{YunEpipelagic}
			that the regular elliptic class corresponds to $\uu_P$ via Lusztig's map.
			The main theorem of \cite{YunKL} tells us the Kazhdan-Lusztig map is a section of Lusztig's map.
			Thus $\cO_\theta$ is in the same conjugacy class as $\uu_P$. 	
			
		\end{itemize}
	\end{rem}
	
	\subsubsection{Counterexamples}\label{sss:tame monodromy counterexamples}
	\begin{itemize}
		\item [(1)]
		Let $\LG=\SO_7$.
		Consider the stable grading of $\Lg=\so_7$ of order $m=2$,
		which has $s_0=0$.
		In this case, $p(\Lg_1^\st)$ contains two $\LG(0)$-orbits,
		where $\uu_P$ is not the open orbit $\cO_\theta$ in $\Lg(1)$
		but rather the smaller orbit.
		Therefore Corollary \ref{c:theta conn 0 monodromy} fails.
		See \S\ref{ss:eg of min orbit conj} for more details.
		
		\item [(2)]
		Let $\LG=\SO_{2n}$. 
		Consider the stable inner grading of $\Lg=\so_{2n}$ with $s_0=0$ and
		of order $m=\frac{2n}{k}$, where $k>2$ is an even divisor of $n$,
		i.e. the third class of \cite[Table 14]{RLYG}.
		In this case, $p(\Lg_1^\st)$ contains more than one $\LG(0)$-orbit,
		so that Corollary \ref{c:theta conn 0 monodromy}.(i) fails.
		Moreover, $\uu_P$ is also not the open orbit $\cO_\theta$,
		but rather the smallest $\LG(0)$-orbit in $\Lg(1)$.
		
		\item [(3)]
		Let $\LG=G_2$.
		Consider the stable grading of $\fg_2$ with $s_0=0$ and of order $m=2$,
		the third class of \cite[Table 7]{RLYG}.
		In this case there are three nonzero $\LG(0)$-orbits in $\Lg(1)$,
		whose stabilizers in $\LG(0)$ have dimensions $0,1,2$ respectively.
		These three orbits all belong to $p(\Lg_1^\st)$.
		However, $\uu_P$ corresponds to the orbit with stabilizer dimension one,
		i.e. neither the open orbit $\cO_\theta$ nor the minimal orbit.
		
		\item [(4)]
		Consider an inner grading of positive rank on $\Lg$ with $s_0>0$
		that is not necessarily stable.
		We compare the Richardson class defined by $\clambda=mx$,
		i.e. the class $\uu_P$ when the grading is stable with $s_0>0$,
		and the orbit $\cO_\theta$.
		Since
		\[
		\Stab_{\hP}(X_1)=\{ug\in\hP|\ u=\exp(\sum_{i\geq1}Y_i),\ Y_i\in\Lg(i),\ [Y_i,X_1]=0;\ g\in\LG(0),\ g\cdot X_1=X_1\},
		\]
		we have
		\begin{align*}
			&\hP\cdot X_1\ \mathrm{is\ open\ dense\ in}\ \fu_{\hP}
			\Leftrightarrow\dim\hP\cdot X_1=\dim\fu_{\hP}\\
			\Leftrightarrow&\dim\Stab_{\hP}(X_1)
			=\dim\LG(0)-\dim\LG(0)\cdot X_1+\sum_{i\geq 1}\dim\Lg(i)-\dim[\Lg(i),X_1]\\
			&\hspace{2.55cm}=\dim\LG(0)+\sum_{i\geq 0}\dim\Lg(i+1)-\dim[\Lg(i),X_1]\\
			&\hspace{2.55cm}=\dim\LG(0)\\
			\Leftrightarrow&[\Lg(i),X_1]=\Lg(i+1),\quad\forall i\geq 0.
		\end{align*}
		
		A necessary condition for the above to hold is that
		\[
		\dim\Lg(i)\geq\dim\Lg(i+1),\quad \forall\ i\geq 0.
		\]
		
		However, this is not true in general\footnote{The inequality is always true for $i=0$, see the proof of \cite[Proposition 3.1]{BC}}. 
		Consider the grading on $\Lg=\Sl_{2n+1}$ defined by the adjoint action of a diagonal matrix $t$ satisfying $\alpha_i(t)=1$, $i\neq n,n+1$; $\alpha_n(t)=\alpha_{n+1}(t)=\exp(2\pi i/3)$. 
		Here $\alpha_i$ are simple roots.
		The order of this grading is $3$. 
		It is easy to see that this grading has positive rank: 
		taking an element $X\in\Lg_1$, 
		$X^3$ is a block diagonal matrix contained in 
		$\gl_n\times \gl_1\times\gl_n$. 
		We can choose $X$ so that $X^3$ has nonzero component in $\gl_1$. Then $X^{3n}\neq 0$, which cannot happen if $X\in\Sl_{2n+1}$ is nilpotent.
		For this grading, $\LG(0)\simeq P(\GL_n\times\GL_1\times\GL_n)$, 
		$\Lg(1)\simeq M_{n\times 1}\oplus M_{1\times n}$, 
		$\Lg(2)\simeq M_{n\times n}$. Thus $\dim\Lg(1)<\dim\Lg(2)$. 
		
		We can see that in the above example where
		$s_0>0$ and the grading has positive rank, 
		any nilpotent class in $\Lg(1)$ is not the Richardson class.
	\end{itemize}

	\subsection{$\Lg=\fsp_{2n}$}\label{ss:0 mon Lg=sp}\mbox{}
	
	Let $V=\bC^{2n}$ be a $2n$-dimensional complex vector space with a fixed basis
	\[
	e_1,e_2,...,e_n,e_{-n},...,e_{-1}.
	\] 
	Then
	\[
	\fsp_{2n}=\{
	\begin{pmatrix}
		A & B\\
		C & D
	\end{pmatrix}
	\big|D=-sA^t s,\ B=sB^t s,\ C=sC^t s
	\},\quad \mathrm{where} \quad
	s=\begin{pmatrix}
		& & & &1\\
		& & &1& \\
		& &\reflectbox{$\ddots$}& & \\
		&1& & & \\
		1& & & &
	\end{pmatrix}.
	\]
	Here $sB^t s$ is the transpose of $B$ with respect to the anti-diagonal. We would later also use $s$ to denote unit anti-diagonal matrices of other sizes. 
	A Cartan subalgebra can be given by
	\[
	\fh=\{\diag(a_1,a_2,...,a_n,-a_n,...,-a_1)\big| a_i\in\bC\}.
	\]
	Let $\chi_i$ be the projection from $\fh$ to $a_i$, 
	then the roots are $\chi_i\pm\chi_j$, $i\neq j$ and $2\chi_i$. 
	If an element $X\in\fsp_{2n}$ is regular semisimple, 
	then after being conjugated into $\fh$, 
	it is nonvanishing under all the roots. 
	Equivalently,
	its $2n$ eigenvalues are all distinct from each other and are nonzero.
	
	We verify the following stronger property:
	\begin{lem}\label{l:Lg=sp X_1 in u_P}
		Let $\Lg=\fsp_{2n}$ and consider any stable grading of order $m$ on $\Lg$.
		Then for any stable vector $X=X_1+X_{1-m}$,
		$\exp(X_1)\in\uu_P$.
	\end{lem}
	\begin{proof}
		Note that $\fsp_{2n}$ has only inner automorphisms.
		By \cite[Table 13]{RLYG}, any stable grading on $\fsp_{2n}$
		satisfies $s_0>0$, 
		where the order $m$ is $2$ times a divisor of $n$.
		
		As said before, the case of $m=2n$ is trivial. 
		Let $k>1$ be a divisor of $n$ such that $m=\frac{2n}{k}$. 
		Consider the order $m$ stable grading of $\Lg=\fsp_{2n}$.
		By \cite[Table 13]{RLYG},
		its Kac coordinates are given by
		\[
		s_i=
		\begin{cases}
			1, \qquad i=kr,\ 0\leq r\leq \frac{n}{k}=\frac{m}{2};\\
			0, \qquad \mathrm{otherwise}. 
		\end{cases}
		\]
		Since $s_0>0$, we have $\Lg(0)=\Lg_0$.
		Then the subgroup $\LG(0)\subset\LG$ with Lie algebra $\Lg(0)$
		coincides with $\LG_0$.
		Explicitly, we have
		\[
		\LG(0)=\{
		\begin{pmatrix}
			M_1&      &       &                  &      & \\
			&\ddots&       &                  &      & \\
			&      &M_{m/2}&                  &      & \\
			&      &       &s(M_{m/2}^t)^{-1}s&      & \\
			&      &       &                  &\ddots& \\
			&      &       &                  &      &s(M_1^t)^{-1}s
		\end{pmatrix}
		\big|M_i\in\GL_k(\bC), 1\leq i\leq \frac{m}{2}
		\},
		\]
		and
		\begin{align*}
			\Lg_1=&\Lg(1)\oplus\Lg(1-m)\\
			=&\{
			\begin{pmatrix}
				0&A_1   &      &         &       &               &      & \\
				&\ddots&\ddots&         &       &               &      & \\
				&      &0     &A_{m/2-1}&       &               &      & \\
				&      &      &0        &A_{m/2}&               &      & \\
				&      &      &         &0      &-sA_{m/2-1}^t s&      & \\
				&      &      &         &       &\ddots         &\ddots& \\
				&      &      &         &       &               &0     &-sA_1^t s\\
				A_m&      &      &         &       &               &      &0
			\end{pmatrix}
			\big|A_i\in M_{k\times k}(\bC), 1\leq i\leq \frac{m}{2}-1;\\
			&\quad A_{m/2},A_m\in M_{k\times k}(\bC),\ A_{m/2}=sA_{m/2}^ts,\ A_m=sA_m^t s
			\},
		\end{align*}
		where $\Lg(1)$ consists of $m-1$ blocks in the upper triangular part, $\Lg(1-m)$ consists of a single block in the lower left corner. 
		For simplicity of notations, in the following we will denote an element of $\LG(0)$ by $g\simeq (M_1,...,M_{m/2})$, and an element of $\Lg_1$ (resp. $\Lg(1)$) by $X\simeq (A_1,...,A_{m/2};A_m)$ (resp. $X_1\simeq(A_1,...,A_{m/2})$).
		
		Let $X=X_1+X_{1-m}\simeq(A_1,...,A_{m/2};A_m)\in\Lg_1$ be a stable vector.
		By Lemma \ref{l:st=rs}, $X$ is regular semisimple,
		thus has only nonzero eigenvalues, i.e. $X$ is invertible. 
		Thus each $A_i$ must be invertible.
		
		Let $g\simeq(M_1,...,M_{m/2})\in\LG(0)$.
		The adjoint action of $\LG(0)$ on $\Lg_1$ is given by
		\[
		\Ad_gX_1\simeq(M_1A_1M_2^{-1},...,M_{m/2-1}A_{m/2-1}M_{m/2}^{-1},M_{m/2}A_{m/2}sM_{m/2}^ts).
		\]
		Since $A_{m/2}=sA_{m/2}^ts$, we see $sA_{m/2}=(sA_{m/2})^t$ is symmetric. 
		Thus under the congruence transformation by some invertible matrix $sM_{m/2}s$, we have $(sM_{m/2}s)sA_{m/2}(sM_{m/2}s)^t=I$, 
		which is equivalent to $M_{m/2}A_{m/2}sM_{m/2}^ts=s$. 
		By inductively choosing $M_i=M_{i+1}A_i^{-1}$,
		any $X_1$ can be conjugated by $g$ to the following matrix:
		\begin{equation}\label{eq:type C O_theta}
			X_1^\circ\simeq(I,...,I,s).
		\end{equation}
		
		Finally, the Jordan form of $X_1^\circ$ consists of $k$ many $m\times m$ Jordan blocks. 
		This matches with the table of $\uu_P$ in \cite[\S4.9]{YunEpipelagic} 
		for $\LG=\Sp_{2n}$.
		We conclude $\exp(X_1)\in\uu_P$.
	\end{proof}

	\subsection{$\Lg=\so_{2n+1}$, $m=n$ is even}\label{ss:0 mon Lg=so_2n+1 m=n even}\mbox{}
	
	Let $V=\bC^{2n+1}$ be a dimension $2n+1$ complex vector space with a fixed basis
	\[
	e_1,e_2,...,e_n,e_0,e_{-n},...,e_{-1}.
	\]
	Denote the Gram matrix by
	\[
	J=\begin{pmatrix}
		& &s\\
		&2& \\
		s& &
	\end{pmatrix},
	\]
	then
	\[
	\SO_{2n+1}=\{M\in\GL_{2n+1}|\ MJM^t=J \},\quad \so_{2n+1}=\{M\in\gl_{2n+1}|\ MJ+JM^t=0 \}.
	\]
	Explicitly:
	\[
	\so_{2n+1}=\{
	\begin{pmatrix}
		A & -sx^t & B\\
		2y & 0     & 2x\\
		C & -sy^t & D
	\end{pmatrix}
	\big|D=-sA^t s,\ B=-sB^t s,\ C=-sC^t s, x,y\in M_{1\times n}(\bC)
	\}.
	\]
	A Cartan subalgebra can be given by
	\[
	\fh=\{\diag(a_1,a_2,...,a_n,0,-a_n,...,-a_1)\big| a_i\in\bC\}.
	\]
	
	Let $\chi_i$ be the projection from $\fh$ to $a_i$. 
	The roots are $\pm\chi_i$, $\chi_i\pm\chi_j$, $i\neq j$. 
	If an element $X\in\so_{2n+1}$ is regular semisimple, 
	then after being conjugated into $\fh$, 
	it is non-vanishing under all the roots. 
	Equivalently, apart from one zero eigenvalue, 
	the rest of its $2n$ eigenvalues are all distinct from each other 
	and are nonzero.
	
	\begin{lem}\label{l:Lg=so_2n+1 X_1 in u_P}
		Let $\Lg=\so_{2n+1}$ and consider the stable grading 
		of even order $m=n$ on $\Lg$.
		Then for any stable vector $X=X_1+X_{1-m}$,
		$\exp(X_1)\in\uu_P$.
	\end{lem}
	\begin{proof}
		By \cite[Table 12]{RLYG}, 
		the Kac coordinates of this grading are given by
		\[
		s_i=
		\begin{cases}
			1, \qquad i=0,1,3,5,...,n-3,n-1;\\
			0, \qquad \mathrm{otherwise}. 
		\end{cases}
		\]
		Thus
		\begin{align*}
			\LG(0)=&\{
			\begin{pmatrix}
				M_1&      &       &         &                  &      & \\
				&\ddots&       &         &                  &      & \\
				&      &M_{m/2}&         &                  &      & \\
				&      &       &M_{m/2+1}&                  &      & \\
				&      &       &         &s(M_{m/2}^t)^{-1}s&      & \\
				&      &       &         &                  &\ddots& \\
				&      &       &         &                  &      &s(M_1^t)^{-1}s
			\end{pmatrix}
			\big|M_1\in\GL_1(\bC); \\
			&\quad M_i\in\GL_2(\bC), 2\leq i\leq \frac{m}{2}; M_{m/2+1}\in\SO_3
			\},
		\end{align*}
		and
		\begin{align*}
			\Lg_1=&\Lg(1)\oplus\Lg(1-m)\\
			=&\{
			\begin{pmatrix}
				0&A_1   &      &         &               &      & \\
				&\ddots&\ddots&         &               &      & \\
				&      &0     &A_{m/2}&               &      & \\
				&      &      &0        &-sA_{m/2}^t s&      & \\
				&      &      &         &\ddots         &\ddots&  \\
				A_{m+1}&      &      &         &               &0     &-sA_1^t s\\
				0&-sA_{m+1}^ts      &      &         &               &      &0
			\end{pmatrix}
			\big|A_i\in M_{2\times 2}(\bC), 2\leq i\leq \frac{m}{2}-1;\\
			&\quad A_1=(a_1,a_2)\in M_{1\times 2}(\bC), A_{m+1}=(b_1,b_2)^t\in M_{2\times 1}(\bC),A_{m/2}\in M_{2\times 3}(\bC)
			\},
		\end{align*}
		where $\Lg(1)$ consists of $m$ blocks in the upper triangular part, $\Lg(1-m)$ consists of $2$ blocks in the lower triangular part. 
		
		For simplicity of notations, 
		in the following we will denote an element of $\LG(0)$ by 
		$g\simeq (M_1,...,M_{m/2+1})$, 
		and an element of $\Lg_1$ (resp. $\Lg(1)$) by 
		$X\simeq (A_1,...,A_{m/2};A_{m+1})$ (resp. $X_1\simeq(A_1,...,A_{m/2})$).
		Denote $I_{2\times 3}=\begin{pmatrix}
			1&0&0\\
			0&0&1
		\end{pmatrix}$.
		
		Let $X=X_1+X_{1-m}\simeq(A_1,...,A_{m/2};A_{m+1})\in\Lg_1$ 
		be a stable vector.
		By Lemma \ref{l:st=rs}, $X$ is regular semisimple,
		thus has exactly one zero eigenvalue and distinct nonzero eigenvalues.
		By $\LG(0)$-conjugation, we may assume $A_{m/2}$ is upper triangular. 
		Note that if any row, the first column, or the third column of $A_{m/2}$ is zero, so is the transpose in $-sA_{m/2}^ts$. 
		In that case $\det(\lambda I-X)$ has at least two zero roots, 
		and $X$ cannot be regular semisimple. 
		Thus $A_{m/2}$ must have rank two
		with its first and third columns being nonzero.
		As the $\LG(0)$-conjugation on $X_1$ acts as $M_{m/2}A_{m/2}M_{m/2+1}^{-1}$ on $A_{m/2}$, 
		we can always reduce $A_{m/2}$ to $I_{2\times 3}$, 
		so we will assume $A_{m/2}=I_{2\times 3}$. 
		Then the $n+1$-th row and column of $X$ are zero.
		Since $X$ has exactly one zero eigenvalue,
		it follows that all the $A_i$'s are of full rank.
		By $\LG(0)$-conjugation, 
		we may assume $A_i=I$, $2\leq i\leq\frac{m}{2}-1$.
		
		For $A_1=(a_1,a_2)$, denote $A_{m+1}=(b_1,b_2)^t$, $X\simeq((a_1,a_2),I,...,I;(b_1,b_2)^t)$. 
		Then
		\[
		\det(\lambda I-X)=\lambda(\lambda^{2n}+2(-1)^{\frac{n}{2}}(a_1b_1+a_2b_2)\lambda^n+(a_1b_1-a_2b_2)^2).
		\]
		
		If $a_2=0$, then 
		$\det(\lambda I-X)=\lambda(\lambda^n+(-1)^{\frac{n}{2}}a_1b_1)^2$. 
		so that $X$ has repeated eigenvalues and cannot be regular semisimple. 
		Thus we must have $a_2\neq0$. 
		Similarly, $a_1\neq0$.
		Using conjugation by $g\simeq(M_1,M_2,...,M_{m/2+1})$ where $M_2=M_3=\cdots=M_{m/2}=\diag(a,a^{-1}),M_{m/2+1}=\diag(a,1,a^{-1})$, 
		we can make $A_1=(1,1)$. 
		We obtain that $X_1$ is always $\LG(0)$-conjugated to the following matrix:
		\begin{equation}\label{eq:type B m=n even O_theta}
			X_1^\circ\simeq((1,1),I,...,I,I_{2\times 3}).
		\end{equation}
		
		Finally, the Jordan form of $X_1^\circ$ has one size $1$ block, one size $n-1$ block, and one size $n+1$ block. 
		This matches with the table in \cite[\S4.9]{YunEpipelagic} for $\LG=\SO_{2n+1}$, $m=n$ even.
	\end{proof}

	\subsection{More examples}\label{ss:s_0=0 ex with corr mon}
	For stable gradings on $\Lg=\so_{2n+1}$ where $s_0=0$,
	the stable grading on $\fg=\fsp_{2n}$ with the same order 
	always has $s_0>0$.
	We give examples of stable vectors $X=X_1+X_{1-m}$
	that satisfies $\exp(X_1)\in\uu_P$.
	However, we cannot yet show this for arbitrary stable vectors.

	\subsubsection{$\Lg=\so_{2n+1}$, $s_0=0$, $k=\frac{2n}{m}>2$ is even}
	From \cite[Tabel 12]{RLYG}, the Kac coordinates are given by
	\[
	s_i=
	\begin{cases}
		1, \qquad i=kr-\frac{k}{2}, 1\leq r\leq \frac{n}{k}=\frac{m}{2};\\
		0, \qquad \mathrm{otherwise}. 
	\end{cases}
	\]
	
	When $m>2$, the shape of $\Lg_1$ is similar as that for 
	$\Lg=\so_{2n+1}$ and $m=n$ even,
	except the sizes of blocks are now given by
	$A_1\in M_{k/2\times k}(\bC)$, 
	$A_i\in M_{k\times k}(\bC)$ for $2\leq i\leq \frac{m}{2}-1$,
	$A_{m/2}\in M_{k\times(k+1)}(\bC)$, 
	$A_{m+1}\in M_{k\times k/2}(\bC)$.
	Let 
	$A_1=(I,0_{k/2\times k/2})$, 
	$A_i=I$ for $2\leq i\leq\frac{m}{2}-1$, 
	$A_{m/2}=(I,(1,0,...,0)^t)$, 
	$A_{m+1}=(I,0_{k/2\times k/2})^t$.
	
	When $m=2$, the only difference is that $A_1=A_{m/2}\in M_{k/2\times(k+1)}(\bC)$, 
	where we let $A_1=(I,(1,0,...,0)^t,0_{k/2\times k/2})^t$.
	
	It is straightforward to check that $X$ is regular semisimple
	and $X_1$ has the same Jordan type as $\uu_P$ 
	given in the table in \cite[\S4.9]{YunEpipelagic}.

	\subsubsection{$\Lg=\so_{2n+1}$, $s_0=0$, $k=\frac{2n}{m}>1$ is odd}
	From \cite[Tabel 12]{RLYG}, the Kac coordinates are given by
	\[
	s_i=
	\begin{cases}
		1, \qquad i=kt-\frac{k-1}{2}, 1\leq t\leq \frac{n}{k}=\frac{m}{2};\\
		0, \qquad \mathrm{otherwise}. 
	\end{cases}
	\]
	
	The shape of $\Lg_1$ is similar as the case of $k$ even, 
	except the sizes of blocks are now given by
	$A_1\in M_{(k+1)/2\times k}(\bC)$,
	$A_i\in M_{k\times k}(\bC)$ for $2\leq i\leq \frac{m}{2}$,
	$A_{m+1}\in M_{k\times (k+1)/2}(\bC)$.
	Let 
	$A_1=(I,0_{(k+1)/2\times(k-1)/2})$,
	$A_i=I$ for $2\leq i\leq \frac{m}{2}$,
	$A_{m+1}=\begin{pmatrix}A\\ B\end{pmatrix}$ where
	$A=(0_{(k-1)/2\times 1},I)\in M_{(k-1)/2\times(k+1)/2}(\bC)$
	and $B=E_{(k+1)/2,1}\in M_{(k+1)/2\times(k+1)/2}(\bC)$.
	
	It is straightforward to check that $X$ is regular semisimple
	and $X_1$ has the same Jordan type as $\uu_P$ 
	given in the table in \cite[\S4.9]{YunEpipelagic}.

	\section{Local monodromy of $\theta$-connections at $\infty$}\label{s:theta conn local monodromy at infty}
	
	\begin{prop}\label{p:theta conn at infty}
		Let $X\in\Lg_1^\st$ be a stable vector for a stable grading of order $m$,
		$\nabla^X$ the associated $\theta$-connection.
		Let $\nabla^X_\infty=\nabla^X|_{D_\infty^\times}$
		be the restriction to $D_\infty^\times=\Spec\ \bC(\!(s)\!)$, $s=t^{-1}$.
		\begin{itemize}
			\item [(i)]
			The slope of $\nabla^X_\infty$ is $\frac{1}{m}$.
			
			\item [(ii)]
			Let $\nabla^{X,\Ad}_\infty$ be the associated $\GL(\Lg)$-connection
			for the adjoint representation. 
			The irregularity of $\nabla^{X,\Ad}_\infty$ is
			\[
			\mathrm{Irr}(\nabla^{X,\Ad}_\infty)=\frac{|\Phi|}{m},
			\]
			where $\Phi$ is the set of roots of $\Lg$.
			
			\item [(iii)]
			Let $u=s^{1/m}$. 
			The canonical form of $\nabla^X_\infty$ is
			\begin{equation}\label{eq:infty Jordan}
				\td-mX\frac{\td u}{u^2}.
			\end{equation}
			
			\item [(iv)]
			The wild inertia group of $\nabla^X_\infty$
			is the smallest torus $S$ whose Lie algebra contains $X$.
			Let $\cT_X$ be the unique maximal torus containing $S$.
			Then a generator of the tame inertia group,
			which is contained in $N_{\LG}(\cT_X)$,
			maps into the unique $\bZ$-regular elliptic class of order $m$
			in the Weyl group $W=N_{\LG}(\cT_X)/\cT_X$.
			
			\item [(v)]
			The connection $\nabla^X_\infty$ is irreducible.
		\end{itemize}
	\end{prop}
	\begin{proof}
		(i):
		Let $s=t^{-1}$ be a coordinate around $\infty$. 
		The restriction of the $\theta$-connection $\nabla^X$ \eqref{eq:theta conn} 
		to $D_\infty^\times$ is
		\[
		\nabla^X_\infty=\td-\sum_k X_ks^{\frac{k-1}{m}}\frac{\td s}{s}=\td-(X_1+X_{1-m}s^{-1})\frac{\td s}{s}.
		\]
		Recall $\clambda=mx$ defines the grading, 
		where $\clambda:\bGm\ra\cT$. 
		Take substitution $s=u^m$ and apply gauge transformation by $\clambda(u^{-1})$, we get
		\begin{equation}\label{eq:pre Jordan form}
			\td-m(X_1+X_{1-m})\frac{\td u}{u^2}-\clambda\frac{\td u}{u},
		\end{equation}
		where $X_1+X_{1-m}=X\in\Lg_1$ is a regular semisimple vector. 
		Thus the maximal slope of $\nabla^X_\infty$ is $\frac{1}{m}$.\\
		
		(ii):
		Let $\ct_X$ be the centralizer of $X$, which is a Cartan subalgebra. 
		Let $\LG[\![u]\!]_1=\ker(\LG[\![u]\!]\xrightarrow{u=0} \LG)$.
		From \cite[\S 9.3 Proposition 2]{BV}, 
		the above connection can be $\LG[\![u]\!]_1$-gauge transformed 
		to a connection of the form
		\[
		\td-mX\frac{\td u}{u^2}+(y+uY(u))\frac{\td u}{u}
		\]
		where $y\in\ct_X$, $Y(u)\in\ct_X[\![u]\!]$. 
		We can eliminate $Y(u)$ using the gauge transform given by 
		$\exp(\int Y(u))\in\cT_X[\![u]\!]$. 
		Note that $\int Y(u)\in u\ct_X[\![u]\!]$, 
		so $\exp(\int Y(u))\in\cT_X[\![u]\!]$ is well defined. 
		Thus $\nabla^X_\infty$ is $\LG(\!(u)\!)$-gauge equivalent 
		to a connection of the following form:
		\begin{equation}\label{eq:infty Jordan 0}
			\td-mX\frac{\td u}{u^2}+y\frac{\td u}{u},\qquad y\in\ct_X.
		\end{equation}
		From the above, 
		since the total multiplicity of the nonzero eigenvalues 
		of the adjoint action of the regular semisimple element $X$
	    on $\Lg$ is $|\Phi|$, 
	    we obtain that the adjoint irregularity of $\nabla^X_\infty$ is
		$\frac{|\Phi|}{m}$.\\
		
		(iii):
		Recall we have shown that $\nabla^X_\infty$ can be gauge transformed 
		by an element $g\in\LG(\bC(\!(u)\!))$ 
		to a canonical form \eqref{eq:infty Jordan 0}:
		\[
		\td-mX\frac{\td u}{u^2}+y\frac{\td u}{u}=\td-Xs^{-1-\frac{1}{m}}\td s+\frac{1}{m}ys^{-1}\td s,\qquad y\in\ct_X.
		\]
		We want to show $y=0$. 
		Applying \cite[\S8.6 Proposition]{BV} to $\ct_X$,
		we can assume $y$ is weakly $\bZ$-reduced in $\ct_X$ 
		in the sense of \cite[\S8.6 Definition]{BV}.  
		Thus we can apply \cite[\S9.8 Proposition]{BV} to conclude that 
		the element $\sigma=g(\zeta_m u)g(u)^{-1}$ is constant, 
		i.e. contained in $\LG$, and it satisfies
		\[
		\sigma^m=1,\quad \Ad_\sigma y=y,\quad \Ad_\sigma X=\zeta_m^{-1}X.
		\]
		
		Let $\gamma=\sigma^{-1}\in\LG$. 
		Since $X$ is an eigenvector of $\gamma$, 
		$\gamma$ normalizes the centralizer $\cT_X$ of $X$, 
		i.e. $\gamma\in N_{\LG}(\cT_X)$.
		Note that the product $\sigma\clambda(\zeta_m)$ fixes $X$,
		thus $\sigma\clambda(\zeta_m)\in\cT_X$.
		Therefore $\gamma=\sigma^{-1}$ has the same image 
		in $W=N_{\LG}(\cT_X)/\cT_X$ as $\clambda(\zeta)$,
		i.e. in the unique $\bZ$-regular elliptic order $m$ class. 
		Therefore $\gamma$ has no nonzero fixed vector in $\ct_X$, 
		which implies that $y=0$.\\
		
		(iv):
		Denote the differential Galois group for a Laurent field $K$ by $\pi_{\diff}(K)$. 
		We have an exact sequence \cite[2.6.1.2]{KatzDGal}
		\[
		1\ra \pi_{\diff}(\bC(\!(u)\!))\ra \pi_{\diff}(\bC(\!(s)\!))\ra \mu_m\ra 1.
		\]
		
		Let $\rho:\pi_{\diff}(\bC(\!(s)\!))\ra\LG$ be the representation 
		corresponding to $\nabla_\infty^X$, 
		then $\rho\big|_{\pi_{\diff}(\bC(\!(u)\!))}$ corresponds to 
		\eqref{eq:infty Jordan}. 
		Let $S$ be the smallest subtorus $S\subset\cT_X$ 
		whose Lie algebra contains $X$. 
		We know from \cite[2.6.4.2]{KatzDGal} that 
		$\rho(\pi_{\diff}(\bC(\!(u)\!)))$ is a connected torus, 
		which is contained in $S$. 
		Also, the Lie algebra of $\rho(\pi_{\diff}(\bC(\!(u)\!)))$ must contain $X$.
		Thus $\rho(\pi_{\diff}(\bC(\!(u)\!)))=S$. 
		
		Moreover, let $\gamma:u\ra\zeta_m u$ be an automorphism of $\bC(\!(u)\!)$ 
		that generates $\mathrm{Gal}(\bC(\!(u)\!)/\bC(\!(s)\!))\simeq\mu_m$. 
		Denote a preimage of $\gamma$ in $\pi_{\diff}(\bC(\!(s)\!))$ 
		still by $\gamma$, 
		then it has adjoint action on $\pi_{\diff}(\bC(\!(u)\!))$. 
		We see $\rho(\pi_{\diff}(\bC(\!(s)\!)))$ is generated by 
		$S$ and $\rho(\gamma)$.
		
		Precomposing 
		$\rho:\pi_{\diff}(\bC(\!(u)\!))\ra\LG$ given by \eqref{eq:infty Jordan}
		with the adjoint action by $\gamma$, we get
		\[
		\td-mX\zeta_m^{-1}\frac{\td u}{u^2}.
		\]
		Since $\rho\circ\Ad_{\gamma}=\Ad_{\rho(\gamma)}\circ\rho:\pi_{\diff}(\bC(\!(u)\!))\ra S\hookrightarrow \LG$, 
		the adjoint action of $\rho(\gamma)$ on $\Lie S$ satisfies
		\[
		\Ad_{\rho(\gamma)}X=\zeta_m^{-1}X.
		\]
		Therefore $\rho(\gamma)\in N_{\LG}(\cT_X)$. 
		Note that $X$ is also an eigenvector with eigenvalue $\zeta_m^{-1}$ 
		for $\clambda(\zeta_m)^{-1}$. 
		Thus $\rho(\gamma)\clambda(\zeta_m)$ fixes $X$, $\rho(\gamma)\clambda(\zeta_m)\in\cT_X$, 
		$\rho(\gamma)$ and $\clambda(\zeta_m)^{-1}$ have the same image 
		in the Weyl group $W=N_{\LG}(\cT_X)/\cT_X$, 
		which belongs to the unique $\bZ$-regular elliptic order $m$ class.\\
		
		(v):
		The differential Galois group of $\nabla^X_\infty$ 
		is generated by the smallest torus $S$ 
		whose Lie algebra contains the regular semisimple element $X$, 
		together with an element $\rho(\gamma)$ that normalizes 
		the maximal torus $\cT_X$ containing $S$, 
		whose image in the Weyl group $N_{\LG}(\cT_X)/\cT_X$ belongs to 
		the $\bZ$-regular elliptic class corresponding to the stable grading $\theta$.
		Then the proof of irreducibility is the same as \cite[Lemma 7]{YiFG},
		as it only relies on the above-mentioned properties. 
	\end{proof}
	
	\begin{rem}
		We learned that essentially the same proof has been independently discovered 
     	and used in the work of Kamgarpour-Sage on Coxeter connections
		 \cite[Theorem 35, Corollary 36]{KSGiff}.
	\end{rem}

	\section{Local opers with slope bounded by $\frac{1}{m}$}\label{s:local opers with slope 1/m}
	An oper is a connection on a $\LG$-bundle together with a Borel reduction
	of the bundle satisfying certain transversality condition\cite{BD}.
	It plays a central role in the geometric Langlands correspondence,
	where the Hecke eigensheaf can be constructed for connections with oper structures, c.f. \cite{BD,Zhu}.
	
	In this section we study opers on the punctured formal disk $D_\infty^\times$
	whose underlying connection is $\nabla^X_\infty$. 
	This will be one of the ingredients in the construction of the Hecke eigensheaf
	of $\nabla^X$.
	
	\subsection{}\label{ss:local opers}
	Let $\Op_{\Lg}(D_\infty)_{\leq 1/m}$ be the space of opers 
	with slopes smaller or equal to $\frac{1}{m}$.
	Since $\nabla^X_\infty$ has slope $\frac{1}{m}$,
	$\chi_\infty\in\Op_{\Lg}(D_\infty)_{\leq 1/m}$.
	This space of opers can be described as a central support.
	Let $\fZ$ be the center of the completed enveloping algebra
	of the affine Kac-Moody algebra $\hg$ at critical level.
	Feigin-Frenkel isomorphism gives
	$\Spec\fZ\simeq\Op_{\Lg}(D_\infty^\times)$.
	Let $\fp,\fp(k)$ be the Lie algebras of 
	the parahoric subgroup $P$ and its Moy-Prasad subgroups $P(k)$, $k\geq0$. 
	Denote $\Vac_{\fp(2)}=\Ind_{\fp(2)+\bC1}^{\hg}\bone$
	and
	$\fZ_{\fp(2)}=\mathrm{Im}(\fZ\rightarrow\End\Vac_{\fp(2)})$.
	\begin{lem}\label{l:oper slope 1/m}
		For an admissible parahoric $P$, 
		$\Spec\fZ_{\fp(2)}$ coincides
		with $\Op_{\Lg}(D_\infty)_{\leq 1/m}$ inside $\Op_{\Lg}(D_\infty^\times)$.
	\end{lem}
	\begin{proof}
		From \cite[Corollary 14]{RLYG}, 
		any stable grading is principal. 
		Thus the lemma follows from \cite[Lemma 16]{Zhu}
		\footnote{An alternative proof can be given by the argument 
			in \cite[Proposition 28.(i)]{KXY}}.
	\end{proof}

	\subsection{}
	We collect some corollaries of results in \cite[\S9]{BV} we will need. 
	A formal $\LG$-connection $\nabla$ on $D_\infty^\times$ 
	can be written as
	\begin{equation}\label{eq:connection form}
		\td+(A_rs^r+A_{r+1}s^{r+1}+\cdots)\td s,\quad A_i\in\Lg,\ r\in\bZ.
	\end{equation}
	
	A canonical form of $\nabla$ is an equation of the following form:
	\begin{equation}\label{eq:canonical form}
		\td+(D_{r_1}s^{r_1}+\cdots+D_{r_N}s^{r_N}+Cs^{-1})\td s,
	\end{equation}
	where $r_1<r_2<\cdots<r_N<-1$ are rational numbers, 
	$D_i,C\in\Lg$ are mutually commutative elements, 
	and $D_i$'s are all semisimple. 
	The main theorem of \cite[\S9.5]{BV} tells us 
	a formal connection $\nabla$ can always be transformed to a canonical form 
	by some element in $\LG(\bC(\!(s^{1/b})\!))$ for some $b\in\bN^+$,
	and the polar part $D_i$'s of the canonical form 
	is unique up to a constant transform.
	
	\begin{lem}\label{l:principal deg of nil}
		Let $\Lg$ be a reductive Lie algebra. 
		For a formal connection $\nabla$ with equation \eqref{eq:connection form}, 
		if $A_r\neq 0$ is nilpotent and $r<-1$, 
		then $r<r_1$ in its canonical form \eqref{eq:canonical form}
	\end{lem}
	\begin{proof}
		This follows from \cite[Proposition 9.4]{BV}. 
		Note that we can always transform the connection 
		to let it satisfy the assumptions of Proposition 9.4 
		using \cite[\S9.3 Proposition 2]{BV}.
	\end{proof}
	
	\begin{prop}\label{p:ss part of top polar coeff}
		Let $\nabla$ be an irregular formal connection 
		with equation \eqref{eq:connection form}, 
		where $r<-1$ and $A_r\neq 0$. 
		Let $A_r=A_s+A_n$ be the Jordan decomposition, 
		where $A_s$ is semisimple, $A_n$ is nilpotent, and $[A_s,A_n]=0$. 
		Assume $A_s\neq 0$. 
		Then in a canonical form \eqref{eq:canonical form} of $\nabla$, 
		$r_1=r$ and $D_{r_1}=A_s$.
	\end{prop}
	\begin{proof}
		Let $\fz(A_s)$ be the centralizer of $A_s$, 
		which is a Levi subalgebra. 
		By \cite[\S9.3 Proposition 2]{BV}, 
		we can transform \eqref{eq:connection form} over $\bC(\!(s)\!)$ 
		without changing $r$ and $A_r$ 
		so that $A_{r+k}\in\fz(A_s)$, $\forall k\geq 1$. 
		Thus we may assume $A_{r+k}\in\fz(A_s)$ from the beginning. 
		Write
		\[
		\nabla=(\td+(A_ns^r+A_{r+1}s^{r+1}+\cdots)\td s)+A_ss^r\td s.
		\]
		Let $Z(A_s)$ be the connected centralizer of $A_s$ in $\LG$, 
		which is a Levi subgroup. 
		Regarding $\td+(A_ns^r+A_{r+1}s^{r+1}+\cdots)\td s$ 
		as a formal $Z(A_s)$-connection, 
		it can be transformed 
		using $Z(A_s)(\bC(\!(s^{1/b})\!))$ for some $b$ 
		to a canonical form:
		\[
		\td+(D_{r_2}s^{r_2}+\cdots+D_{r_N}s^{r_N}+Cs^{-1})\td s,
		\quad D_{r_i},\ C\in\fz(A_s).
		\]
		Apply Lemma \ref{l:principal deg of nil} 
		to $\fz(A_s)$ and the above connection, 
		we get $r<r_2$. 
		Since $Z(A_s)(\bC(\!(s^{1/b})\!))$ fixes $A_s$, 
		we obtain that $\nabla$ is equivalent to
		\[
		\td+(A_ss^r+D_{r_2}s^{r_2}+\cdots+D_{r_N}s^{r_N}+Cs^{-1})\td s.
		\]
		This is a canonical form of $\nabla$, 
		which proves the proposition.
	\end{proof}

	\subsection{}
	From \cite[Proposition 15]{Zhu}
	\footnote{An alternative proof can be given by the argument 
		in \cite[Proposition 28.(i)]{KXY}}, 
	we have a commutative diagram
	\begin{equation}\label{eq:comm diag}
	\begin{tikzcd}
		U(V_P)^{L_P} \arrow[r,hook] \arrow[d,hook] &\fZ_{\fp(2)} \arrow[d,hook]\\
		U(V_P) \arrow[r,hook] &\End(\Vac_{\fp(2)})
	\end{tikzcd}
	\end{equation}
	
	\begin{prop}\label{p:compatibility of characters at infty}
		Let $\chi_\infty$ be any oper structure on $\nabla^X|_{D_\infty^\times}$,
		which factors through $\fZ_{\fp(2)}$ by Lemma \ref{l:oper slope 1/m}.
		Let $\phi$ correspond to $X$ under \eqref{eq:GIT isom}.
		The following two characters coincide:
		\[
		U(V_P)^{L_P}\ra\fZ_{\fp(2)}\xrightarrow{\chi_\infty}\bC,\qquad U(V_P)^{L_P}\ra U(V_P)\xrightarrow{\phi}\bC.
		\]
	\end{prop}
	\begin{proof}
		Let $\nabla_\infty$ be the underlying connection of 
		$\chi_\infty$. 
		We can write it in the oper canonical form:
		\begin{equation}\label{eq:oper form}
			\nabla_\infty=\td-(p_{-1}+\sum_{i=1}^n v_i(s)p_i)\td s,\quad v_i(s)=\sum_j v_{ij}s^{-j-1},\ v_{ij}\in\bC.
		\end{equation}
		Since the underlying connection $\nabla_\infty$ has slope $\frac{1}{m}$, 
		we know from \cite[Lemma 16]{Zhu} that $v_{ij}\neq0$ only if
		$j\leq d_i+\lfloor\frac{d_i}{m}\rfloor-1$.
		
		Under the Feigin-Frenkel isomorphism $\fZ\simeq\Fun\Op_{\Lg}(D_\infty^\times)$, 
		the coefficients $v_{ij}$ map to a set of Segal-Sugawara operators $S_{ij}$
		that topologically generate $\fZ$. 
		From the proof of \cite[Proposition 28.(i)]{KXY}, 
		we can see
		\begin{equation}\label{eq:U(V_P)^L_P generators}
			U(V_P)^{L_P}=\bC[S_{i,d_i+\frac{d_i}{m}-1},m|d_i].
		\end{equation}
		Identify $\cfc\simeq p_{-1}+\sum_i\bC p_i$ via Kostant section.
		Under the composition
		$V_P^*/\!\!/L_P\hookrightarrow\fc^*\simeq\cfc\simeq p_{-1}+\sum_i\bC p_i$,
		the generators $\{S_{i,d_i+\frac{d_i}{m}-1},m|d_i\}$
		map to the coefficients of $\{p_i,m|d_i\}$ in the Kostant section.
		
		Now taking substitution $s=u^m$, 
		replacing $j$ with $j+d_i-1$ in the oper form \eqref{eq:oper form}, 
		and applying the gauge transform by $\crho(u^{m+1})$, 
		we obtain equation
		\begin{equation}
			\begin{split}
				&\td-m(u^{-2}p_{-1}+\sum_{i=1}^n\sum_{j\leq\lfloor\frac{d_i}{m}\rfloor}v_{i,j+d_i-1}u^{d_i-mj-2}p_i)\td u-\crho\frac{\td u}{u}\\
				=&\td-m(p_{-1}+\sum_{m|d_i}v_{i,d_i+\lfloor\frac{d_i}{m}\rfloor-1}p_i)\frac{\td u}{u^2}+A\frac{\td u}{u},\quad A\in\Lg[\![u]\!].
			\end{split}
		\end{equation}
		Since $\nabla_\infty$ has maximal slope $\frac{1}{m}$, 
		we know from the slope formula for opers that 
		at least one of 
		$v_{i,d_i+\lfloor\frac{d_i}{m}\rfloor-1}p_i$ 
		in the above is nonzero. 
		Therefore the leading term
		$-m(p_{-1}+\sum_{m|d_i}v_{i,d_i+\lfloor\frac{d_i}{m}\rfloor-1}p_i)$ 
		is non-nilpotent. 
		Applying Proposition \ref{p:ss part of top polar coeff}, 
		we obtain that in the canonical form of $\nabla^X$ in terms of $u$, 
		the order $r=-2$, 
		and the coefficient of the leading term 
		is conjugated to the semisimple part of
		$-m(p_{-1}+\sum_{m|d_i}v_{i,d_i+\lfloor\frac{d_i}{m}\rfloor-1}p_i)$.
		
		On the other hand, 
		we have shown that the canonical form of 
		$\nabla^X|_{\bC(\!(u)\!)}$ 
		is given by \eqref{eq:infty Jordan}, 
		i.e. $\td-mX\frac{\td u}{u^2}$. 
		Thus the semisimple part of
		$p_1+\sum_{m|d_i}v_{i,d_i+\lfloor\frac{d_i}{m}\rfloor-1}p_i$
		is conjugated to $X$.
		Denote $\pi:\Lg\ra\Lg/\!/\LG=\cfc$. 
		Then
		$\pi(p_{-1}+\sum_{m|d_i}v_{i,d_i+\lfloor\frac{d_i}{m}\rfloor-1}p_i)=\pi(X)$. 
		Now the character
		$U(V_P)^{L_P}\ra\fZ_{\fp(2)}\xrightarrow{\chi_\infty}\bC$
		maps generators 
		$S_{i,d_i+\lfloor\frac{d_i}{m}\rfloor-1}$ 
		to $v_{i,d_i+\lfloor\frac{d_i}{m}\rfloor-1}$, 
		thus corresponds to $\pi(X)$ under $V_P^*/\!/L_P\hookrightarrow\fc^*\simeq\cfc$. 
		This is exactly the definition of $\phi$. The proof is complete.
	\end{proof}

	\section{A space of global opers}\label{s:global opers}
	Since $\theta$-connections have no obvious oper structures 
	except in the case of Frenkel-Gross connection where $m=h$,
	we cannot directly apply the method of Beilinson-Drinfeld to
	construct their Hecke eigensheaves.
	Instead, we will first consider a space of global opers,
	for which we construct nonzero Hecke eigensheaves.
	We will later show that the underlying connections of these opers 
	are exactly $\theta$-connections.
	
	\subsection{Central support of parabolic local opers}
	Recall that $P\subset G(\cO_0)$.
	Its Lie algebra is $\fp\subset\fg(\cO)$.
	Thus $\fp+\bC1$ is a subalgebra of $\hg=\fg(\!(t)\!)+\bC1$.
	Denote $\Vac_\fp=\Ind_{\fp+\bC1}^{\hg}\bone$,
	with generator $v_0=1\otimes\bone\in\Vac_\fp$.
	Denote
	$\fZ_\fp=\mathrm{Im}(\fZ\rightarrow\End(\Vac_\fp))$,
	$\Op_{\Lg,P}(D_0)_0	=\Spec\fZ_\fp$.
	Here $\End(\Vac_\fp)\simeq\Vac_\fp^P\subset\Vac_\fp$ by acting on $v_0$,
	so we regard $\fZ_\fp$ as a subspace of $\Vac_\fp$.
	As in the proof of Proposition \ref{p:compatibility of characters at infty},
	the center $\fZ$ is generated by Segal-Sugawara operators $S_{i,j}$.
	The associated graded of $\fZ$ is the space of functions on local Hitchin base
	$\Hit(D^\times)\simeq\bigoplus_{i=1}^n\omega_F^{d_i}$, 
	c.f. Appendix \ref{s:appendeix}.
	The symbol $\overline{S_{i,j}}$ of $S_{i,j}$ 
	is the coefficient of $t^{-j-1}(\td t)^{d_i}$.
	
	\begin{prop}\label{p:parabolic central support}
		The image $\fZ_\fp$ is freely generated by the images of 
		$S_{i,j}$ for $1\leq i\leq n$, $j\leq d_i-\lceil\frac{d_i}{m}\rceil-1$.
		The kernel of $\fZ\rightarrow\fZ_\fp$
		is generated by $S_{i,j}$ for $1\leq i\leq n$, $j\geq d_i$,
		together with $S_{i,j}-f_{i,j}(S_{p,q})$ for $1\leq i\leq n$,
		$d_i-\lceil\frac{d_i}{m}\rceil\leq j\leq d_i-1$,
		where $f_{i,j}$ is a polynomial of $S_{p,q}$ 
		for $q\leq d_i-\lceil\frac{d_i}{m}\rceil-1$
	\end{prop}
	\begin{proof}
		The associated graded of the action map
		$\fZ\rightarrow\fZ_\fp\hookrightarrow\Vac_\fp^P\hookrightarrow\Vac_\fp$
		is the functions on local Hitchin map \eqref{eq:local Hithin map}
		\[
		\Fun\Hit(D^\times)\rightarrow\gr\fZ_\fp\hookrightarrow(\Fun\fp^\perp)^P\hookrightarrow\Fun\fp^\perp
		\]
		where $\gr\fZ_\fp$ is with respect to the filtration on $\Vac_\fp$.
		
		By Proposition \ref{p:p local Hitchin image},
		the image of the local Hitchin map is $\Fun\Hit(D)_\fp\subset\gr\fZ_\fp$
		\eqref{eq:Hit_p}.
		By Proposition \ref{p:inv poly on p(1)},
		$\Fun\Hit(D)_\fp\subset\gr\fZ_\fp\subset(\Fun\fp^\perp)^P$
		is an isomorphism.
		We get $\gr\fZ_\fp\simeq\Fun\Hit(D)_\fp$.
		As a result, $\fZ_\fp$ is freely generated by those $S_{i,j}v_0$
		whose symbols generate $\Fun\Hit(D)_\fp$.
		The images $S_{i,j}v_0$ of other $S_{i,j}$, 
		i.e. those with $j\geq d_i-\lceil\frac{d_i}{m}\rceil$,
		are polynomials of $S_{i,j}v_0$ for $j\leq d_i-\lceil\frac{d_i}{m}\rceil-1$.
		
		It remains to show $S_{i,j}v_0=0$ for $j\geq d_i$.
		In fact, since $P$ contains Iwahori $I$,
		the action of $\fZ$ on $\Vac_\fp$
		factors through $\Vac_{\mathrm{Lie}(I)}=\Ind_{\mathrm{Lie}(I)+\bC1}\bone$.
		The central support of $\Vac_{\mathrm{Lie}(I)}$ corresponds to 
		some opers with	regular singularities,
		see \cite[Corollary 13.3.2]{FGLocal},
		so that the $S_{i,j}$ for $j\geq d_i$ already acts as zero on 
		$\Vac_{\mathrm{Lie}(I)}$.
		The proof is complete.
	\end{proof}

	\subsection{Global oper space}
	In the following we replace the parahoric $P$ 
	with its opposite parahoric	$P^\opp$.
	Proposition \ref{p:parabolic central support} holds for $P^\opp$ 
	with the same proof.
	Consider the following space of opers:
	\begin{equation}\label{eq:Op_cG}
		\Op_{\cG}:=\Op_{\Lg}(\bGm)\times_{\Op_{\Lg}(D_0^\times)}\Op_{\Lg,P^\opp}(D_0)_0\times_{\Op_{\Lg}(D_\infty^\times)}\Op_{\Lg}(D_\infty)_{\leq 1/m}.
	\end{equation}
	
	The restriction to $\infty$ gives a morphism 
	$\Op_{\cG}\rightarrow\Op_{\Lg}(D_\infty)_{\leq 1/m}$. 
	As in \eqref{eq:comm diag}, 
	we have a morphism 
	$\Fun(V_P^*/\!/L_P)\simeq U(V_P)^{L_P}
	\rightarrow\fZ_{\fp(2)}\simeq\Fun\Op_{\Lg}(D_\infty)_{\leq 1/m}$. 
	Composing these two morphisms, we get a morphism
	\begin{equation}\label{eq:Op_cG to V//L}
		p:\Op_{\cG}\rightarrow V_P^*/\!/L_P.
	\end{equation}
	
	\begin{prop}\label{p:Op_cG to V//L}
		The morphism $p$ in \eqref{eq:Op_cG to V//L} is an isomorphism.
		Explicitly, there exists polynomials $f_{i,j}$ 
		for $1\leq i\leq n, 0\leq j\leq \lceil\frac{d_i}{m}\rceil-1$ 
		in variables $v_{i,d_i+\frac{d_i}{m}-1}$, $1\leq i\leq n$, $m|d_i$, 
		such that
		\begin{equation}\label{eq:global opers}
			\begin{split}
				\Op_{\cG}
				=\{&\td+(t^{-1}p_{-1}+\sum_{m|d_i}v_{i,d_i+\frac{d_i}{m}-1}t^{\frac{d_i}{m}-1}p_i+\sum_{i=1}^n\sum_{0\leq j\leq \lceil \frac{d_i}{m}\rceil-1}f_{i,j}(v_{k,d_k+\frac{d_k}{m}-1},m|d_k)t^{j-1}p_i)\td t\\
				&\mid v_{i,d_i+\frac{d_i}{m}-1}\in\bC,m|d_i\}.
			\end{split}
		\end{equation}
	\end{prop}
	\begin{proof}
		Recall from the discussion in \S\ref{s:local opers with slope 1/m} that $\fZ_{\fp(2)}$ is a free algebra generated by a collection of Segal-Sugawara operators $S_{i,j}$, 
		and $U(V_P)^{L_P}$ is generated by a subset of those generators $S_{i,d_i+\frac{d_i}{m}-1}$ 
		where $d_i$ is divisible by $m$ as in \eqref{eq:U(V_P)^L_P generators}. 
		Thus the map 
		$\Op_{\cG}\rightarrow\Spec\fZ_{\fp(2)}\twoheadrightarrow\Spec U(V_P)^{L_P}$ is first writing a global oper $\chi$ 
		into its oper canonical form at $\infty$, 
		i.e. $\chi_\infty=d+(p_{-1}+\sum_i v_i(s)p_i)ds$, 
		$v_i(s)=\sum_j v_{ij}s^{-j-1}$, $s=t^{-1}$,
		then taking the coefficients 
		$v_{i,d_i+\frac{d_i}{m}-1}$ for $m\mid d_i$. 
		Therefore to show $p$ is isomorphic, 
		it suffices to show \eqref{eq:global opers}.
		
		We begin with
		\[
		\Op_{\Lg}(\bGm)=\{\td+(p_{-1}+\sum_{i=1}^n\bC[t,t^{-1}]p_i)\td t\}.
		\]
		For $\nabla=\td+(p_{-1}+\sum_{i=1}^n\sum_j v_{i,j}t^{-j-1}p_i)\td t\in\Op_{\Lg}(\bGm)$ 
		to sit inside $\Op_{\Lg,P^\opp}(D_0)_0$,
		by Proposition \ref{p:parabolic central support} we have
		$j\leq d_i-1$, and for $d_i-\lceil\frac{d_i}{m}\rceil\leq j\leq d_i-1$,
		$v_{i,j}=f_{i,j}(v_{p,q},q\leq d_i-\lceil\frac{d_i}{m}\rceil-1)$.
		Next we change coordinate to $s=t^{-1}$:
		\[
		\nabla=\td-(p_{-1}s^{-2}+\sum_{i=1}^n(\sum_{j\leq d_i-\lceil\frac{d_i}{m}\rceil-1}v_{i,j}s^{j-1}+\sum_{d_i-\lceil\frac{d_i}{m}\rceil\leq j\leq d_i-1}f_{i,j}(v_{p,q})s^{j-1})p_i)\td s.
		\]
		After gauge transforming the above into oper canonical forms on $D_\infty^\times$
		and applying the oper slope formula as in the proof of Proposition \ref{p:compatibility of characters at infty},
		we see $\nabla\in\Op_{\Lg}(D_\infty)_{\leq 1/m}$ if and only if
		$v_{i,j}=0$ for $j\leq d_i-\lceil\frac{d_i}{m}\rceil-2$.
		Finally we switch back to coordinate $t$ and apply transform by
		$\crho(t)$ and $\exp(-p_1/2)$, 
		we obtain \eqref{eq:global opers} where the $j$ in $f_{i,j}$ are re-indexed. 
	\end{proof}

	\section{Proof of Theorem \ref{t:main}}\label{s:proof of main thm}
	\subsection{Set up the proof}
	Let $X\in\Lg_1^\st$ be any stable vector, 
	with associated $\theta$-connection $\nabla^X$. 
	Let $\phi\in V_P^{*,\st}$ be a stable functional
	that matches with $X$ under \eqref{eq:st bij}.
	Let $\cA_\phi$ be the irreducible holonomic $D$-module
	on $\Bun_{\cG}$ that is $(V_P,\cL_\phi)$-equivariant,
	where $\cG$ is the group scheme with level structure
	$P^\opp$ at $0$ and $P(2)$ at $\infty$.
	Then $\cA_\phi$ is Hecke eigen with an eigenvalue $\cE_\phi$.
	Recall that $\cE_\phi$ is regular singular at $0$
	with monodromy class $\uu_P$.
    
    By Proposition \ref{p:Op_cG to V//L}, 
    there exists a unique $\nabla_\phi\in\Op_{\cG}$ that corresponds to $\phi$.
	We first construct a nonzero complex $\cA$ 
	of holonomic $D$-modules on $\Bun_{\cG}$
	that is $(V_P,\cL_\phi)$-equivariant 
	and is a Hecke-eigen with eigenvalue $\nabla_\phi$.
	Once such a complex $\cA$ has been constructed, 
	by rigidity of the epipelagic automorphic datum, 
	$\cA$ is a complex of shifts of copies of $\cA_\phi$.
	Thus its eigenvalue must be the underlying connection of $\nabla_\phi$.
	Then it will be easy to deduce that $\nabla_\phi$ is isomorphic to
	the $\theta$-connection $\nabla^X$ 
	using physical rigidity \cite{HJrigid}.
	
	\subsection{The Hecke eigen complex}\label{ss:Hecke eigensheaf I^opp}	
	\subsubsection{The construction}
	Let $\chi$ be the character of $\Fun\Op_{\cG}$ corresponding to $\nabla_\phi$.
    Consider
	\begin{equation}\label{eq:cA global}
		\cA\simeq\omega_{\Bun_{\cG}}^{-1/2}
		\otimes_{\cO_{\Bun_{\cG}}}(\cD_{\Bun_{\cG}}'
		\otimes^{\mathrm{L}}_{\Fun\Op_\cG\otimes_{U(V_P)^{L_P}}U(V_P)}\bC_{\chi,\phi}),
	\end{equation}
    where $U(V_P)$ acts on $\cD_{\Bun_{\cG}}'$ via \eqref{eq:comm diag}
    and the localization functor.
	By the proof of \cite[Theorem 8]{Zhu}, 
	$\cA$ is Hecke-eigen over $\bGm$ with eigenvalue $\nabla_\phi$.
	Observe that by Proposition \ref{p:Op_cG to V//L},
	$\cA\simeq\omega_{\Bun_{\cG}}^{-1/2}
	\otimes_{\cO_{\Bun_{\cG}}}(\cD_{\Bun_{\cG}}'
	\otimes^{\mathrm{L}}_{U(V_P)}\bC_\phi)$.
	
	\subsubsection{Nonvanishing}
	\begin{lem}
		The complex $\cA$ is nonzero.
	\end{lem}
	\begin{proof}	
		Recall the open embedding of relevant orbit
		$V_P=P(1)/P(2)\hookrightarrow\Bun_{\cG}$.
		It suffices to show 
		\[
		H^0(\cA|_{V_P})=\omega_{V_P}^{-1/2}
		\otimes_{\cO_{V_P}}(\cD_{V_P}'
		\otimes_{U(V_P)}\bC_\phi)
		\]
		is nonzero.
		We prove this by showing that 
	    $\cD_{V_P}'$ is flat over $U(V_P)$.
	    In fact, the associated graded map is between function rings of
	    the moment map
		$T^*(V_P)\simeq V_P\times V_P^*$ to $V_P^*$, 
		which is smooth.
		Thus $\cD_{V_P}'$ is flat over $U(V_P)$, and $\cA\neq 0$.
	\end{proof}
	
	\subsubsection{Holonomicity}
	Denote by $i$ the closed embedding into $\Bun_{\cG}$
	of the complement of open substack $V_P$. 
	We have a distinguished triangle
	\[
	i_*i^!\cA\ra\cA\ra j_*j^*\cA\xrightarrow{[1]}.
	\]
	
	Since $\cA$ is $(V_P,\cL_\phi)$-equivariant, $i^!\cA=0$. 
	Thus $\cA\simeq j_*j^*\cA$.
	Since $j^*\cA$ is $(V_P,\cL_\phi)$-equivariant, it must be holonomic.
	So $\cA\simeq j_*j^*\cA$ is holonomic and $(V_P,\cL_\phi)$-equivariant, 
	thus must be a complex of copies of the irreducible $D$-module $\cA_\phi$.

	\subsection{Identification of connections}
	By \cite[Theorem 4.5]{YunEpipelagic},
	the eigenvalue $\nabla_\phi$ of $\cA_\phi$
	must be regular singular at $0$ with monodromy class $\uu_P$.
	Thus by Corollary \ref{c:theta conn 0 monodromy}.(i),
	$\nabla_\phi$ has the same local monodromy as $\nabla^X$ at $0$.
	 
	By Proposition \ref{p:theta conn at infty} and
	Proposition \ref{p:compatibility of characters at infty},
	$\nabla_\phi$ and $\nabla^X$ are both isoclinic at $\infty$ with slope $\frac{1}{m}$ and leading term $X$.
	Therefore they have the same connection canonical forms at $\infty$.
	By \cite[Lemma 18]{YiToral}, 
	$\nabla_\phi$ and $\nabla^X$ are isomorphic over $D_\infty^\times$.
	
	Thus $\nabla_\phi$ and $\nabla^X$ have the same local monodromies.
	By \cite[Theorem 1.2.1]{HJrigid}, 
	that such local monodromies	are physically rigid.
	We conclude that $\nabla^X\simeq\nabla_\phi$,
	completing the proof of Theorem \ref{t:main}.

	\subsection{Remark on the failure of global flatness}\label{ss:moment not flat}
	Unlike \cite[Lemma 17, Lemma 18]{Zhu},  
	the whole moment map $\mu:T^*\Bun_\cG\ra V_P^*$ can fail to be flat
	when $P$ is strictly larger than $I$. 
	In fact, we shall see that $\mu$ is flat if and only if $P=I$.
	
	Let $\cG(0,1)=\cG(P^\opp,P(1))$. 
	Then $\Bun_\cG\ra\Bun_{\cG(0,1)}$ is a $V_P$-bundle, 
	and $\dim\Bun_{\cG(0,1)}=0$. 
	We have Hamiltonian reduction
	$$
	T^*\Bun_{\cG(0,1)}\simeq\mu^{-1}(0)/V_P.
	$$
	Suppose $\mu:T^*\Bun_\cG\ra V_P^*$ is flat. 
	Since we have open embedding $V_P\hookrightarrow\Bun_\cG$, 
	there is open embedding $T^*V_P\ra T^*\Bun_\cG$. 
	We can see that generically, fibers of $\mu$ contain $V_P$ as an open subscheme, 
	thus have dimension $\dim V_P$. 
	If $\mu$ is flat, its fiber dimension is constant, so that 
	$\dim\mu^{-1}(0)=\dim V_P$ and 
	$\dim T^*\Bun_\cG=0=2\dim\Bun_{\cG(0,1)}$. 
	So the flatness of $\mu$ is equivalent to 
	the goodness of $\Bun_{\cG(0,1)}$. 
	However, the latter is not true if $P$ is larger than $I$. 
	To see this, consider Birkhoff decomposition:
	\[
	\Bun_{\cG(0,1)}
	=\bigsqcup_{\tw\in W_P\backslash\tW} P^-\backslash P^-\tw L_P P(1)/P(1).
	\]
	
	The goodness condition is equivalent to the condition that
	\[
	\mathrm{codim}\{y\in\Bun_{\cG(0,1)}\mid\dim\Aut(y)=n\}\geq n,\quad\forall n>0.
	\]
	
	In the stratum $P^-\backslash P^-\tw L_P P(1)/P(1)$, 
	all the objects have the same automorphism group $P^-\cap\tw P(1)\tw^{-1}$, whose dimension is close to $\ell(\tw)$ 
	up to a constant determined by $G$ and $P$. 
	Thus there are only finitely many strata
	whose automorphism groups have dimension
	equal to the same given number $n$.
	Denote
	\[
	Y_n=\{\tw\in W_P\backslash\tW\mid\dim(P^-\cap\tw P(1)\tw^{-1})=n\}.
	\] 
	Any $P^-\backslash P^-\tw L_P P(1)/P(1)$ can be regarded as 
	a $P^-\cap\tw P(1)\tw^{-1}$-gerbe over $\tw^{-1}P^-\tw\cap L_P\backslash L_P$. 
	Thus the inequality of goodness becomes
	\begin{align*}
		&0-(\max_{\tw\in Y_n}\dim(\tw^{-1}P^-\tw\cap L_P\backslash L_P)-n)\geq n\\
		\Leftrightarrow&\dim(\tw^{-1}P^-\tw\cap L_P)=\dim L_P,\quad\forall \tw\in Y_n\\
		\Leftrightarrow&\tw L_P\tw^{-1}\subset P^-,\quad \forall \tw\in Y_n.
	\end{align*}
	
	Thus for the goodness to hold, 
	the above should hold for all but finitely many $\tw$. 
	However, when $P\neq I$, we can construct infinitely many $\tw$
	that fail the above relation.
	In this case there exists nonzero affine root $\tilde{\alpha}\in\Phi(L_P)$. 
	Write $\tilde{\alpha}=\alpha+n_\alpha\cdot\delta$. 
	Let $x$ be the point in the apartment defining $P$. 
	For $\tw=w\cdot v\in W\ltimes X_*(T)$, 
	if the above inclusion holds, then
	\[
	\tw\tilde{\alpha}(x)=(w\alpha)(x)+n_\alpha+\alpha(v)\leq 0.
	\]
	Fix $w$. 
	Since $\alpha\neq0$, there exists $v_0$ such that $\alpha(v_0)>0$. 
	For any $v=\ell v_0$, $\ell>\!\!>0$, the above inequality fails. 
	Thus $\Bun_{\cG(0,1)}$ is not good, and $\mu$ is not flat.

	\section{Epipelagic Langlands parameters}\label{s:applications on epipelagic L param}
	In positive characteristic case,
	the automorphic representation generated by the
	Frobenius trace function of the 
	unique irreducible $(V_P,\cL_\phi)$-equivariant perverse sheaf on $\Bun_{\cG}$
	has epipelagic representation of \cite{RY} as its local component at $\infty$.
	Thus the monodromy at $\infty$ of the Hecke eigenvalue $\cE_\phi$ 
	of the $(V_P,\cL_\phi)$-equivariant $D$-module $\cA_\phi$
	is the de Rham analog of the epipelagic Langlands parameter.
	
	\subsection{}
	In \cite[\S7.1]{RY}, 
	Reeder and Yu made predictions on the epipelagic Langlands parameters.
	We can deduce the following 
	de Rham analog of their predictions from our main result:
	\begin{prop}\label{p:epipelagic L parameter}
		For $\phi\in V_P^{*,\st}$,
		let $\nabla_\infty=\cE_\phi|_\infty$ be the 
		de Rham epipelagic Langlands parameter.
		\begin{itemize}
			\item [(i)] 
			The monodromy group of $\nabla_\infty$ has zero fixed vectors on $\Lg$.
			
			\item [(ii)] 
			The slope of $\nabla_\infty$ is $\frac{1}{m}$,
			and its adjoint irregularity is
			$\mathrm{Irr}(\nabla_\infty^{\Ad})=\frac{|\Phi|}{m}$.
			
			\item [(iii)] 
			The image of wild inertia maps into 
			a maximal torus $\cT$ 
			such that a generator of the tame inertia group maps to the order $m$
			$\bZ$-regular elliptic class in $W=N_{\LG}(\cT)/\cT$.
		\end{itemize}
	\end{prop}
	\begin{proof}
		By Theorem \ref{t:main},
		we have $\cE_\phi\simeq\nabla^X$ for some $X\in\Lg_1^\st$.
		
		For (ii), the slope and adjoint irreducibility of $\nabla^X|_\infty$
		is known from (i), (ii) of Proposition \ref{p:theta conn at infty}.
		
		For (iii), the wild inertia image 
		and generator of tame inertia of $\nabla^X|_\infty$ 
		are given in Proposition \ref{p:theta conn at infty}.(iv).
		
		For (i), any vector in $\Lg$ fixed by the wild inertia 
		commutes with $X$, thus is in the Cartan subalgebra $\Lie(\cT_X)$.
		The elements in $\Lie(\cT_X)$ that are further fixed by the tame generator
		must be zero, since the action of tame generator 
		is elliptic.
	\end{proof}
    
    \begin{rem}
    	In \cite[Proposition 5.2]{YunEpipelagic},
    	Yun proved the cohomological rigidity of 
    	the eigenvalue of epipelagic automorphic datum,
    	assuming (i), (ii) in the Proposition \ref{p:epipelagic L parameter}.
    	Thus we can now remove the assumptions in Yun's proof,
    	obtaining an independent proof of cohomological rigidity 
    	from the one in \cite{Chen}.
    \end{rem}
	
	\subsection{}
	In \cite{FuGuEuler},
	Fu and Gu computed the Euler characteristic of $\cE_\phi^{\mathrm{St}}$,
	where $G$ is a classical group,
	and $\mathrm{St}$ is the standard representation of $\LG$.
	From the isomorphism $\nabla^X\simeq\cE_\phi$,
	we deduce the following de Rham analog of the part of 
	\cite[Theorem 1.2, Theorem 1.3]{FuGuEuler}
	under our assumptions on the stable grading:
	\begin{prop}
		Let $G$ be of type $B_n$, $C_n$, or $D_n$.
		Then we have
		\[
		-\chi_c(\cE_\phi^\mathrm{St})=
		\begin{cases}
			\frac{2n}{m},\qquad G\ \text{is type}\ B_n,\ C_n,\\
			\frac{2n}{m},\qquad G\ \text{is type}\ D_n,\ X\ \text{has only nonzero eigenvalues},\\
			\frac{2n-2}{m}, \quad G\ \text{is type}\ D_n,\ X\ \text{has exactly}\ 2n-2\ \text{nonzero eigenvalues}.
		\end{cases}
		\]
	\end{prop}
	\begin{proof}
		From \cite[Lemma 3.1]{FuGuEuler}, 
		$-\chi_c(\cE_\phi^\mathrm{St})=\mathrm{Irr}_\infty(\cE_\phi^\mathrm{St})$.
		By Proposition \ref{p:theta conn at infty}.(iii),
		$\mathrm{Irr}_\infty(\cE_\phi^\mathrm{St})$ is 
		the number of nonzero eigenvalues of regular semisimple element $X$ 
		on $\mathrm{St}$ divided by $m$.
		For $G$ of type $B_n$ or $C_n$, regular semisimple elements 
		have $2n$ nonzero eigenvalues on $\mathrm{St}$.
		For $G$ of type $D_n$, regular semisimple elements have $2n-2$ or $2n$
		nonzero eigenvalues on $\mathrm{St}$.
		The statement follows.
	\end{proof}

	\section{The Conjecture for $s_0=0$}\label{s:conj for s_0=0}
	For $\cE_\phi$  to be matched with $\nabla^X$,
	their monodromies at $0$ must match,
	i.e. $X_1\in u_{\hat{P}}$.
	However, we have seen in \S\ref{sss:tame monodromy counterexamples}
	that this can fail when $s_0=0$.
	In the following we propose a modification of Theorem \ref{t:main}
	and exhibit an evidence.
	\subsection{A conjectural description of degree one component of stable vectors}\label{ss:gen conj}
	Recall the projection $p:\Lg_1\ra\Lg(1)$. 
	Recall that the image $\pi(\Lg_1^\st)$ always contains 
	the open $\LG(0)$-orbit $\cO_\theta$. 
	When $s_0>0$, the image consists of this single orbit. 
	However, when $s_0=0$, the image can contain more than one $\LG(0)$-orbit. 
	Note $\Lg_0=\Lg(0)\oplus\Lg(m)\oplus\Lg(-m)$.
	The following conjectural description of the smallest orbit in the image generalizes Corollary \ref{c:theta conn 0 monodromy} to the case of $s_0=0$.
	
	\begin{conj}\label{c:0 monodromy s_0=0}
		Let $\theta$ be a stable grading on $\Lg$, with associated parahoric subgroup $P$.
		\begin{itemize}
			\item [(i)] There is a $\LG(0)$-orbit $\cO_P\subset p(\Lg_1^\st)$ satisfying $\exp(\cO_P)\subset\uu_P$.
			\item [(ii)] For any $X\in\Lg_1^\st$, there exists $g\in\exp(\Lg(m))\subset\LG_0$ such that $p(\Ad_g X)\in\cO_P$.
		\end{itemize}
	\end{conj}
	
	In view of the above conjecture, we formulate the following generalization of Theorem \ref{t:main}:
	
	\begin{conj}\label{c:gen main}\
		Let $\theta$ be a stable inner grading of $\Lg$.
			Let $\phi\in V_P^{*,\st}$ be a stable functional. 
			Assume Conjecture \ref{c:0 monodromy s_0=0} is true,
			so that there exists stable vector $X=X(\phi)$ 
			that matches with $\phi$ under \eqref{eq:GIT isom} 
			such that $p(X)\in\cO_P$.
			Then there is an isomorphism of $\LG$-connections: 
			\begin{equation}\label{eq:gen main isom}
				\cE_\phi\simeq\nabla^X.
			\end{equation}
	\end{conj}
	
	\begin{cor}\label{cor:conn form well defn}
		For any $X\in\Lg_1^\st$, 
		then construction of the connection $\nabla^X$ 
		depends only on the $\LG_0$-orbit of $X$ up to isomorphism.
	\end{cor}

	\subsection{An example of Conjecture \ref{c:0 monodromy s_0=0} for $\LG=\SO_7$}\label{ss:eg of min orbit conj}
	Consider the order $m=2$ stable grading on $\so_7$
	belonging to the fourth class of \cite[Table 12]{RLYG},
	where $s_0=0$. 
	In this case, $\LG(0)\simeq\GL_2\times\SO_3$, $\LG_0\simeq\SO_4\times\SO_3$, and $\Lg(1)\simeq M_{2\times 3}(\bC)$. 
	We can easily verify the following fact:
	\begin{lem}\label{l:two orbits SO7}
		\[
		p(\Lg_1^\st)=\cO_\theta\sqcup\cO_P,\qquad \mathrm{where}\ \exp(\cO_P)\subset\uu_P.
		\]
	\end{lem}
	
	Precisely, for $X_1\in p(\Lg_1^\st)$, 
	$X_1$ has rank 4 and $(X_1)^2$ has rank either one or two. 
	The orbit $\cO_\theta$ consists of those whose square have rank two; 
	$\cO_P$ are those whose square have rank one,
	which has the same Jordan type $(2,2,3)$ as $\uu_P$
	given in \cite[\S4.9]{YunEpipelagic}.
	Note that nonzero entries in $\Lg(1)$
	are spanned by entries $\bC E_{ij}$ for $1\leq i\leq 2, 3\leq j\leq 5$
	and $3\leq i\leq 5,6\leq j\leq 7$.
	Thus nonzero entries of $p(\Ad_g X)^2$ support in the 2-by-2 block
	on the upper right corner. 
	Let $q$ be the projection to the 2-by-2 block 
	on the upper right corner. 
	Note that a nonzero 2-by-2 matrix has rank one 
	iff it has vanishing determinant.
	To sum up, we conclude that 
	\[
	\forall X\in\Lg_1^\st,\ p(X)\in\cO_P\ \text{iff}\ \det(q(p(X)^2))=0.
	\]
	
	Observe that $\Lg(m)=\Lg(2)$ consists of the root space of highest root, 
	which we denote by $\gamma$. 
	Explicitly, $\Lg_\gamma$ consists of matrices with nonzero elements 
	in the 2-by-2 block on the upper right corner, 
	where this block is a diagonal matrix of trace zero. 
	Denote by $E_\gamma$ the root vector 
	whose only nonzero entry in the first row equals to $1$. 
	
	The rest of Conjecture \ref{c:0 monodromy s_0=0} 
	amounts to the following statement:
	\[
	\forall X\in\Lg_1^\st\ \text{with}\ p(X)\in\cO_\theta,\
	\exists a\in\bC\ \text{s.t.}\
	\det(q(p(\Ad_{\exp(aE_\gamma)}X)^2))=0.
	\] 
	
	Since the adjoint action of $\LG(0)$ preserves $\Lg(m)$, 
	we may fix $p(X)$. 
	Assume
	\begin{equation}
		X=\begin{pmatrix}
			0&0&1&0&0&0&0\\
			0&0&0&0&1&0&0\\
			-b_6&-b_5&0&0&0&-1&0\\
			-2b_4&-2b_3&0&0&0&0&0\\
			-b_2&-b_1&0&0&0&0&-1\\
			0&0&b_1&b_3&b_5&0&0\\
			0&0&b_2&b_4&b_6&0&0\\
		\end{pmatrix}.
	\end{equation}
	
	For $X$ to be stable, 
	which is equivalent to being regular semisimple, 
	we look at its characteristic polynomial:
	\begin{equation}\label{eq:char poly SO7}
		\det(\lambda I+X)
		=\lambda[\lambda^6-2(b_1+b_6)\lambda^4+((b_1+b_6)^2-4b_2b_5-4b_3b_4)\lambda^2+4((b_1+b_6)b_3b_4-b_5b_4^2-b_2b_3^2)].
	\end{equation}
	For $X$ to be regular semisimple, 
	the above polynomial needs to have one zero root 
	and six distinct nonzero roots.

	On the other hand,
	\[
	q(p(\Ad_{\exp(aE_\gamma)}X)^2)
	=\begin{pmatrix}
		2ab_5+2a^2(b_1b_5+b_3^2)&1+a(b_1-b_6)-a^2(b_1b_6+2b_3b_4+b_2b_5)\\
		1+a(b_1-b_6)-a^2(b_1b_6+2b_3b_4+b_2b_5)&-2ab_2+2a^2(b_2b_6+b_4^2)
	\end{pmatrix}
	\]
	with
	\begin{equation}\label{eq:det 2by2 SO7}
		\begin{split}
			\det(q(p(\Ad_{\exp(aE_\gamma)}X)^2))
			=&[4(b_1b_5+b_3^2)(b_2b_6+b_4^2)-(b_2b_5+2b_3b_4+b_1b_6)^2]a^4\\
			&+[4(b_5(b_2b_6+b_4^2)-b_2(b_1b_5+b_3^2))+2(b_1-b_6)(b_2b_5+2b_3b_4+b_1b_6)]a^3\\
			&+[2(b_2b_5+2b_3b_4+b_1b_6)-4b_2b_5-(b_1-b_6)^2]a^2\\
			&-2(b_1-b_6)a-1.
		\end{split}
	\end{equation}
	
	To verify the conjecture,
	for any given $b_i$'s such that 
	the polynomial \eqref{eq:char poly SO7} has distinct roots, 
	we need to find $a\in\bC$ such that 
	the above polynomial \eqref{eq:det 2by2 SO7} vanishes. 
	Suppose this cannot be done.
	Then all the coefficients of $a^k$ in \eqref{eq:det 2by2 SO7} are zero.
	We show that in this case 
	polynomial \eqref{eq:char poly SO7} will have repeated roots. 
	
	Vanishing of the coefficients of \eqref{eq:det 2by2 SO7} means
	\[
	\begin{cases}
		b_1=b_6,\\
		b_2b_5=b_1^2+2b_3b_4,\\
		b_5b_4^2=b_2b_3^2,\\
		2b_1b_2b_3^2=3b_3^2b_4^2+2b_1^2b_3b_4.
	\end{cases}
	\]
	
	First, it is not hard to see that if $b_4=0$, 
	then we deduce from the above that $b_2b_3^2=0,b_2b_5=b_1^2$. 
	Then \eqref{eq:char poly SO7} becomes $\lambda^5(\lambda^2-4b_1)$,
	 which has repeated zero roots. 
	Thus we may assume $b_4\neq 0$. 
	Similarly, we may assume $b_1\neq 0$. 
	Then we can solve from the above that
	\[
	\begin{cases}
		b_1^2=-\frac{9}{4}b_3b_4,\\
		b_2=-\frac{3b_4^2}{4b_1},\\
		b_5=-\frac{3b_3^2}{4b_1}.
	\end{cases}
	\]
	Take these back into \eqref{eq:char poly SO7}, we obtain
	\[
	\lambda^{-1}\det(\lambda I+X)=(\lambda^2)^3-4b_1(\lambda^2)^2+\frac{16}{3}b_1^2\lambda^2-\frac{64}{27}b_1^3.
	\]
	Take substitution $\lambda^2=b_1t/3$, we obtain the following polynomial:
	\[
	t^3-12t^2+48t-64.
	\]
	Recall that a cubic polynomial $t^3+bt^2+ct+d$ has repeated roots 
	iff its discriminant 
	$\Delta=b^2c^2-4c^3-4b^3d-27d^2+18bcd$ vanishes. 
	This is indeed the case for the above polynomial.  
	This completes the verification of Conjecture \ref{c:0 monodromy s_0=0} 
	in this case.

	\appendix
	\section{Hitchin image}\label{s:appendeix}
	Let $P$ be a parahoric subgroup of the loop group 
	corresponding to a principal grading $\theta$, 
	with Moy-Prasad filtration $P(k)$, $k\geq0$.
	Let $\fp,\fp(k)$ be the Lie algebras of $P,P(k)$.
	As before, we denote $V_P=P(1)/P(2)$ and $L_P=P/P(1)$.
	
	\subsection{Local Hitchin image}
	Let $\fp(k)^\perp\subset\fg(\!(t)\!)\td t$ 
	be the $\cO$-lattice that is orthogonal to $\fp(k)$,
	as in \cite[\S4.3]{Zhu}.
	Consider the local Hitchin map for $\fp(k)$ 
	as in \cite[Proposition 10]{Zhu}:
	\begin{equation}\label{eq:local Hithin map}
		\fp(k)^\perp\ra\Hit(D^\times)\simeq\bigoplus_i\omega_F^{d_i},
	\end{equation}
	where $F=\bC(\!(t)\!)$,
	and the last isomorphism is via invariant polynomials defined by 
	Kostant section $p_{-1}+\sum_i\bC p_i\simeq\fg/\!/G$. 
	In the notation of \S\ref{ss:local opers},
	$\gr\Vac_{\fp(2)}\simeq\Fun\fp(2)^\perp$,
	so that $\gr\fZ_{\fp(2)}$ is isomorphic to 
	the functions on the closure of the image of the above map for $k=2$.
	
	By \cite[Proposition 10]{Zhu}, 
	the above map factors through
	\begin{equation}
		\Hit(D)_{\fp(k)}:=\bigoplus_i\omega_{\cO}^{d_i}(d_i-\lceil\frac{d_i(1-k)}{m}\rceil).
	\end{equation}
	
	In particular, we have
	\begin{equation}\label{eq:Hit_p}
		\Hit(D)_\fp=\bigoplus_i\omega_{\cO}^{d_i}(d_i-\lceil\frac{d_i}{m}\rceil),\quad \Hit(D)_{\fp(2)}=\bigoplus_i\omega_{\cO}^{d_i}(d_i+\lfloor\frac{d_i}{m}\rfloor).
	\end{equation}
	Zhu showed in \cite[Corollary 12, Proposition 13]{Zhu} that 
	for $k=1,2$, $\fp(k)^\perp$ maps onto $\Hit(D)_{\fp(k)}$. 
	Also, it is well known that for $P=I$, 
	$\fI^\perp$ maps onto $\Hit(D)_\fI$.
	
	\begin{prop}\label{p:p local Hitchin image}
	The image closure $\mathrm{Im}(\fp^\perp\ra\Hit(D^\times))$ is $\Hit(D)_\fp$.  
	\end{prop}

	\begin{proof}
		It suffices to make slight modification 
		to the proof of \cite[Proposition 13]{Zhu}. 
		Note that the inclusion of image into $\Hit(D)_\fp$ 
		is known from \cite[Proposition 10]{Zhu}. 
		It remains to show surjectivity.
		
		Since $P$ corresponds to a principal grading $\theta$, 
		there exists a $\theta$-adapted principal $\Sl_2$-triple $\{e,f,h\}$, 
		i.e. $e\in\fg_1, f\in\fg_{-1},h\in\fg_0$. 
		Let $u$ be a $m$-th root of $t$, $t=u^m$. 
		Let $x_P$ be the barycenter of the facet defining $P$, $\eta=m x_P\in X_*(T)$. 
		The adjoint action of $\eta(u)$ gives isomorphism
		\[
		\fp^\perp\simeq\fp(1)\frac{\td t}{t}\simeq\prod_{j\geq 1}\fg_j\otimes u^j\frac{\td u}{u}.
		\]
		Let $q_1,...,q_n$ be a homogeneous basis of $\fg^f$, 
		such that $q_j\in\fg_{-(d_j-1)}$. 
		Recall we define the splitting of Hitchin base using Kostant section. 
		We can write an arbitrary element of $\Hit(D)_\fp$ as
		\[
		\sum_i t^{\lceil\frac{d_i}{m}\rceil-d_i}c_i(t)(\td t)^{d_i}=\sum_i u^{m\lceil\frac{d_i}{m}\rceil}c_i(u^m)(\frac{\td u}{u})^{d_i},\qquad c_i(t)\in\bC[\![t]\!]. 
		\]
		By $e+\fg^f\simeq\fc$, the above element has a preimage 
		$X\frac{\td t}{t}$ under Hitchin map of the form
		\[
		X=e+\sum_i u^{m\lceil\frac{d_i}{m}\rceil}c_i(u^m)q_i.
		\]
		
		Our goal is to find a preimage in $\fp^\perp$. 
		However, $X\frac{\td t}{t}$ is not in this space. 
		It suffices to show that we can conjugate $X$ into $\fp(1)\simeq\prod_{j\geq 1}\fg_j\otimes u^j$. 
		Observe that $h$ has integral eigenvalues on $\fg$, 
		thus it can be given by a cocharacter $\frac{h}{2}:\bGm\rightarrow G^{ad}$. Let $g=\frac{h}{2}(u)\in G^{ad}(\!(u)\!)$.
		Then
		\[
		\Ad_g X=ue+\sum_i u^{m\lceil\frac{d_i}{m}\rceil-d_i+1}c_i(u^m)q_i
		\]
		which belongs to $\prod_{j\geq 1}\fg_j\otimes u^j$.
	\end{proof}
	
	\begin{rem}
		In fact, the image of local Hitchin map 
		has been computed in \cite{BK} 
		when $P$ is of parabolic type and $G$ is either of classical type 
		or of type $G_2$. 
		In \emph{loc. cit.}, 
		the Hitchin map is computed using the characteristic polynomial. 
		Let $m_1\leq m_2\leq\cdots\leq m_n$ be fundamental degrees 
		with multiplicity of Levi $L_P$. 
		Here $1$ is also counted as trivial degree, 
		with possible multiplicity. 
		Then apart from some abnormal cases in type D, 
		the local Hitchin image is as follows:
		\begin{thm}[Theorem 7 \cite{BK}]\label{t:BK parabolic Hitchin image}
			In terms of coefficients of characteristic polynomial, 
			the closure of local Hitchin image is given by $\bigoplus_i\omega_{\cO}^{d_i}(d_i-m_i)$.
		\end{thm}
		
		Observe that when $G$ is of type $A,B,C,G_2$, 
		under the standard representation, 
		the Kostant section of $G$ embeds into that of $\GL_N$. 
		Moreover, for a vector $X=p_{-1}+\sum_ia_ip_i$ 
		in the Kostant section of $\gl_N$, we have
		\[
		\det(\lambda I+X)=\lambda^N+\sum_ic_ia_i\lambda^{N-d_i},
		\]
		where $c_i$ are nonzero constants. 
		Thus the Hitchin image in terms of Kostant sections 
		or characteristic polynomials 
		are the same in these situations. 
		
		In view of the above discussion, 
		when $G$ is of type $A,B,C,G_2$ and $P$ is a parabolic preimage, 
		to verify Proposition \ref{p:p local Hitchin image}, 
		it suffices to check $m_i=\lceil\frac{d_i}{m}\rceil$, 
		which can be done case-by-case using the tables in \cite[\S7]{RLYG}.
		
		The above might be true for any principal parahoric subgroup. 
		But beyond this it should not hold: 
		the Hitchin image in the conjecture depends only on the order $m$, 
		but there are parahorics with the same order $m$ 
		whose Levi quotients have different set of fundamental invariants, 
		thus have different Hitchin image 
		in the case of Theorem \ref{t:BK parabolic Hitchin image}.  
	\end{rem}
	
	\subsection{Global Hitchin base}
	The above calculation works equally when we replace $\fp$ with 
	the Lie algebra $\fp^\opp$ of the opposite parahoric $P^\opp$.
	Define $\Hit(D_0)_{\fp^\opp}$ in the same way as for $\fp$.
	Consider the following global Hitchin base:
	\begin{equation}\label{eq:global Hit}
		\Hit_\cG(\bP^1):=\Hit(\bGm)\times_{\Hit(D_0^\times)}\Hit(D_0)_{\fp^\opp}\times_{\Hit(D_\infty^\times)}\Hit(D_\infty)_{\fp(2)}.
	\end{equation}
	
	Denote 
	\begin{equation}\label{eq:S_m}
		S_m:=\{1\leq i\leq n \ |\ d_i\equiv 0\!\!\!\mod m\}.
	\end{equation}
	
	Combining Proposition \ref{p:p local Hitchin image} 
	and \cite[Proposition 13]{Zhu}, 
	we get:
	\begin{equation}
		\begin{split}
			\Hit_\cG(\bP^1)
			&=\bigoplus_i\Gamma(\bP^1,\omega^{d_i}((d_i-\lceil\frac{d_i}{m}\rceil)\cdot 0+(d_i+\lfloor\frac{d_i}{m}\rfloor)\cdot\infty))\\
			&\simeq\bigoplus_i\Gamma(\bP^1,\cO(\lfloor\frac{d_i}{m}\rfloor-\lceil\frac{d_i}{m}\rceil))\simeq\bA^{|S_m|}.
		\end{split}
	\end{equation}
	
	Moreover, recall from \cite[Proposition 14]{Zhu} 
	that we have surjective morphism
	\[
	\Hit(D_\infty)_{\fp(2)}\ra\bigoplus_{i\in S_m}\omega_{\cO}^{d_i}(d_i+\frac{d_i}{m})/\omega_{\cO}^{d_i}(d_i+\frac{d_i}{m}-1)\simeq V_P^*/\!/L_P.
	\]
	Combining the  discussion above, 
	we obtain the following corollary of Proposition \ref{p:p local Hitchin image}:
	
	\begin{cor}\label{c:global Hit=V//L}
		The composition $\Hit_\cG(\bP^1)\ra\Hit(D_\infty)_{\fp(2)}\ra V_P^*/\!/L_P$ is an isomorphism.
	\end{cor}

	\subsection{Invariant functions}
	We first establish a lemma on principal grading.
	Let $\fg=\oplus_i\fg_i$ be a principal grading of order $m$, 
	$G_0\subset G$ the subgroup with Lie algebra $\fg_0$. 
	Let $\{e,f,h\}$ be a $\theta$-adapted principal $\Sl_2$-triple as before. 
	Let $q_i$ be a homogeneous basis of $\fg^f$. 
	Denote $\fk_1:=e+\sum_{m|d_i}\bC q_i=(e+\fg^f)\cap\fg_1$.
	\begin{lem}\label{l:principal grading}
		\begin{itemize}
			\item [(i)] For any $X\in e+\fg^f$, we have
			\begin{equation}\label{eq:Kos decomp}
				[\fg,X]\oplus\fg^f=\fg.
			\end{equation}
			
			\item [(ii)] For any $i$ and $X_1\in\fk_1$, we have
			\begin{equation}\label{eq:g_i decomp}
				[\fg_{i-1},X_1]\oplus\fg_i^f=\fg_i.
			\end{equation}
		\end{itemize}
	\end{lem}
	\begin{proof}
		(i): Note that for regular element $X\in e+\fg^f$, $\dim[\fg,X]+\dim\fg^f=\fg$. Thus it suffices to show $[\fg,X]\cap\fg^f=0$. Consider morphism
		$$
		\gamma:G\times(e+\fg^f)\rightarrow\fg\xrightarrow{\pi}\fg/\!\!/G,\quad (g,Y)\mapsto\Ad_g Y\mapsto\pi(\Ad_g Y)=\pi(Y).
		$$	
		The tangent map at $(1,X)$ is
		$$
		d\gamma_{(1,X)}:\fg\times\fg^f\rightarrow\fg\rightarrow T_{\pi(X)}(\fg/\!\!/G),\quad d\gamma_{(1,X)}(Y,Z)=d\pi([Y,X]+Z).
		$$
		For any $A\in[\fg,X]\cap\fg^f$, assume $A=[Y,X]\in\fg^f$. Then $d\gamma_{(1,X)}(Y,-A)=d\pi(A-A)=0$.
		
		On the other hand, since $\gamma$ is constant on $G$ and is an isomorphism when $g\in G$ is fixed, we have $d\gamma_{(1,X)}(Y,0)=0$ and $d\gamma_{(1,X)}|_{(0,\fg^f)}$ is bijective. So $d\gamma_{(1,X)}(Y,-A)=d\gamma_{(1,X)}(0,-A)=0$ implies $A=0$. Therefore $[\fg,X]\cap\fg^f=0$.   
		
		(ii): \eqref{eq:g_i decomp} is immediate from letting $X=X_1$ in \eqref{eq:Kos decomp} and taking $\zeta_i$-eigenspace of $\theta$.
	\end{proof}

	\begin{prop}\label{p:inv poly on p(1)}
		Let $P$ be a principal parahoric subgroup associated to the inner principal grading of order $m$.
        The natural restriction 
        $(\Fun\fg(\cO))^{G(\cO)}\twoheadrightarrow(\Fun\fp(1))^P$ 
        is surjective. 
        As a result, the following embedding is an isomorphism:
			\begin{equation}\label{eq:Fun p(1)^P}
				\Fun\Hit(D)_\fp\xrightarrow{\sim}(\Fun\fp^\perp)^P\simeq(\Fun\fp(1))^P,
			\end{equation}
			where the second isomorphism is given by $\fp^\perp\simeq\fp(1)\frac{dt}{t}$ induced by residue pairing and Killing form.
	\end{prop}

	\begin{rem}\mbox{}	
		For $P=G(\cO)$ or Iwahori subgroup $I$, 
		the above is proved in \cite[Theorem 3.4.2, Lemma 8.1.1]{FrenkelBook}.
		The method of our proof is completely different. 
	\end{rem}
	
	\begin{proof}
		We follow the notation in the proof of Proposition \ref{p:p local Hitchin image}. Recall from there that we have the following Kostant section:
		$$
		(ue+\sum_i u^{m\lceil\frac{d_i}{m}\rceil-d_i+1}c_i(u^m)q_i)\frac{du}{u}\hookrightarrow
		\fp(1)\frac{du}{u}\simeq(\prod_{j\geq 1}\fg_j\otimes u^j)\frac{du}{u}\twoheadrightarrow
		\Hit(D)_{\fp}.
		$$
		The above composition is an isomorphism.
		Denote
		\begin{equation}\label{eq:fk fk_1}
			\fk:=ue+\sum_i u^{m\lceil\frac{d_i}{m}\rceil-d_i+1}c_i(u^m)q_i,\quad \fk_1=e+\sum_{m|d_i}\bC q_i\subset\fg_1,\quad \fk_i=\sum_{m|i+d_j-1}\bC q_j=\fg_i\cap\fg^f.
		\end{equation}
		Thus the composition of the following is isomorphic:
		$$
		\Fun\Hit(D)_\fp\rightarrow(\Fun\fp(1)\frac{du}{u})^P\rightarrow\Fun\fk\frac{du}{u}.
		$$
		Therefore the second map in above given by restriction to $\fk$ is surjective. To prove part (i) of the proposition, it remains to show the second map is injective, for which it suffices to show the conjugation subspace $\Ad_P\fk$ is dense in $\fp(1)$. To this end, we show the following:
		
		\emph{Claim}: For any $i\geq 1$, $(\Ad_{L_P}\fk_1+\fp(2))/\fp(i+1)\subset(\Ad_P\fk+\fp(i+1))/\fp(i+1)$ in $\fp(1)/\fp(i+1)$. 
		
		By \cite[Theorem 3.6]{Panyushev}, 
		$\Ad_{L_P}\fk_1$ gives all the regular elements in $\fg_1$,
		thus is open dense in $\fg_1$.
		Therefore $\Ad_{L_P}\fk_1+\fp(2)$ is open dense in $\fp(1)$.
		The denseness of $\Ad_P\fk\subset\fp(1)$ follows from the claim. 
		
		We prove the claim by induction. When $i=1$, it is trivial.
		
		Assume the claim for index smaller than $i$. 
		For $i$, given any element $Y\in(\Ad_{L_P}\fk_1+\fp(2))/\fp(i+1)$, 
		write $Y=Y_{\leq i-1}+Y_i$ 
		where $Y_i\in u^i\fg_i$, $Y_{\leq i-1}$ 
		is the complement in section of $\fp(1)/\fp(i)$. 
		By induction hypothesis, 
		there exists $X\in\fk$ and $g=g_0g_{i-1}\in P=L_P P(1)$, 
		$g_0\in L_P, g_{i-1}\in P(1)$, 
		such that $\Ad_g X=Y_{\leq i-1}+Y_i'$ for some $Y_i'\in\fp(i)/\fp(i+1)$. Write $X=X_1+X_{>1}$ where $X_{>1}\in\fp(2)$, $X_1\in\fk_1 u$. Let $h=\exp(Z)\in P(i-1)$ where $Z\in u^{i-1}\fg_{i-1}$. 
	    It suffices to find $Z$ so that 
	    $Y\in\Ad_h\Ad_g(X+u^i\fk_i )
	    =\Ad_h\Ad_gX+\Ad_{g_0}\fk_i u^i\subset\fp(1)/\fp(i+1)$, 
	    for which it suffices to find $Z\in u^{i-1}\fg_{i-1}\simeq\fg_{i-1}$
	    such that
	    $Y_i\in Y_i'+[Z,\Ad_{g_0}X_1]+\Ad_{g_0}\fk_i\subset\fg_i$.
	    This follows from the conjugation by $g_0\in G_0$ of \eqref{eq:g_i decomp}:
		\begin{equation}\label{eq:recursive eq in g_i}
			[\fg_{i-1},\Ad_{g_0}X_1]+\Ad_{g_0}\fk_i=\fg_i,\quad \forall X_1\in\fk_1,g_0\in G_0.
		\end{equation}
		The claim is proved.
	\end{proof}
	
	\begin{rem}
		In the above proof, when $P=I$, the dense subspace $\Ad_P\fk\subset\fp(1)$ can be alternatively described as follow. Let $\fn^\reg\subset\fn$ be the regular nilpotent Borel orbit in the nilpotent radical $\fn\subset\fb$, then $\fk\subset\fI(1)^\circ=\fn^\reg+t\fg[\![t]\!]=\fI(1)\cap J\fg^\reg$, and $\Ad_I\fk\subset\fI(1)^\circ$. Also, it has been proved in \cite[Lemma 8.1.1]{FrenkelBook} that fibers of $\fI(1)^\circ$ under local Hitchin morphism consist of a single $I$-orbit. As each fiber contains a unique element of the Kostant section $\fk$, the opposite inclusion follows.
		
		We conjecture that for $P\subset G(\cO)$, similar description of $\Ad_P\fk\subset\fp(1)$ can be given by replacing $\fn^\reg$ with the Richardson class.
	\end{rem}

\end{document}